\documentclass[12pt]{article}

\usepackage{amsthm,amsfonts,amsmath,amscd,amssymb,times,epsfig,verbatim}
\usepackage[all]{xy}
\usepackage{enumerate,array}

\newtheorem{thm}{Theorem}
\newtheorem{prop}{Proposition}
\newtheorem{lem}{Lemma}
\newtheorem{cor}{Corollary}

\theoremstyle{remark}
\newtheorem{rem}{Remark}
\theoremstyle{definition}
\theoremstyle{axiom}

\newtheorem{defn}{Definition}
\newtheorem{ex}{Example}

\newcommand{\C}{\mathbb{C}}

\newcommand{\R}{\mathbb{ R}}

\newcommand{\Z}{\mathbb{ Z}}

\title{A resolution of  singularities\\  for the    orbit spaces $G_{n,2}/T^n$ }

\author{Victor M.~Buchstaber and Svjetlana Terzi\'c}

\begin{document}

\maketitle

\begin{abstract}
The problem of the description of the orbit space $X_{n} = G_{n,2}/T^n$ for the standard action of the torus $T^n$ on a complex  Grassmann 
manifold $G_{n,2}$ is widely known and it appears in diversity of mathematical questions. A point $x\in X_{n}$ is said to be a critical point if the stabilizer of its corresponding orbit is nontrivial. In this paper, the notion of singular points  of $X_n$ is introduced  which opened the new approach to this problem. It  is showed that for $n>4$ the set of critical points $\text{Crit}X_n$ belongs to our  set of singular points $\text{Sing}X_{n}$, while  the case $n=4$ is somewhat special for which $\text{Sing}X_4\subset \text{Crit}X_4$,  but there are  critical points which are not singular. 

The central result of this paper is the construction of the smooth manifold $U_n$ with corners, $\dim U_n = \dim X_n$ and an explicit description of the projection $p_{n} : U_{n}\to X_{n}$ which in the defined sense resolve all singular points of the space $X_n$. Thus, we obtain the  description of the orbit space $G_{n,2}/T^n$  combinatorial structure.   Moreover, the $T^n$-action on $G_{n,2}$ is a seminal example of complexity $(n-3)$ - action. Our results demonstrate the method for general description of orbit spaces for torus actions of positive complexity.
\footnote{MSC 2000:  57S25, 57N65, 53D20,  14M25, 52B11, 14B05; keywords:  Grassmann manifold, torus action,  chamber decomposition, orbit space, universal space  of parameters}

\end{abstract}

\section{Introduction}
\subsection{Formulation and the history of the problem}
The studying of the  orbit space $X_n= G_{n,2}/T^n$ of  a complex Grassmann manifold $G_{n,2}$ of the two-dimensional complex planes in $\C ^{n}$ by the standard action of the compact torus $T^n$ is of wide mathematical interest for a while from many aspects:  of algebraic topology, algebraic geometry,  theory of group actions, matroid theory, combinatorics.   This interest  was stimulated  by the works   
Gel'fand-Serganova~\cite{GS}, Goresky-MacPherson~\cite{GM}, Gel'fand-MacPherson~\cite{GeM} in which the $(\C ^{\ast})^{n}$-action on  general Grassmannian $G_{n,k}$, $1< k<n-1$  being  an extension of $T^n$-action, was studied. The results of these papers suggest that the case $k=2$ needs to be studied taking into account the specialty of the spaces $G_{n,2}$.   A little later Kapranov~\cite{Kap} related  $T^n$-action on $G_{n,2}$ with the notion of Chow quotient from algebraic geometry for which he showed  to be isomorphic  to  the Grotendick-Knudsen moduli space $\overline{M}_{0,n}$ of stable  $n$ -  pointed rational  curves of genus zero. These works attracted a lot attention which resulted in series  of papers on related subjects,~\cite{Keel},~\cite{KeelTev},~\cite{LL} etc. 

The complex projective space $\C P^{n-1} = G_{n,1}$ with the standard action of the algebraic torus $(\C ^{\ast})^{n}$ and the induced action of the compact torus $T^n\subset (\C ^{\ast})^{n}$ is   the key object in toric geometry and toric topology respectively, see~\cite{BP}. It is well known that the orbit space 
$G_{n,1}/T^n$ can be identified with the simplex $\Delta ^{n-1}$ which is the  standard example of a smooth manifold with corners. Our studying of an orbit space $X_n = G_{n,2}/T^n$  with its canonical moment map $G_{n,2}\to X_n\to \Delta _{n,2}$, where $\Delta _{n,2}$ is a  hypersimplex has been motivated by the natural problem of  extension of the methods of  toric topology to the case of torus actions of positive complexity.  More recently in~\cite{BT-1} and~\cite{BT-2} the orbit spaces  $X_4=G_{4,2}/T^4 $ and $X_5=G_{5,2}/T^5$ are explicitly described, while an extensive general   study of $T^n$-action on $G_{n,k}$ for $k\geq 2$  is done in~\cite{BT-2},~\cite{BT-2nk} in the context of the theory of $(2n,k)$-manifolds. 

 It is proved in~\cite{BT-1} that the space $X_4$  is  homeomorphic to the sphere $S^5$. According to~\cite{S} the sphere  $S^5$  has a unique smooth structure.  Nevertheless, the space  $X_4$ is a topological  sphere, that is it is not possible to introduce a smooth structure on $X_4$ such that the natural projection  $G_{4,2}\to G_{4,2}/T^4$  is  a smooth map.  Precisely, it is shown in~\cite{BT-1} that $S^5$ has cone-like singularities at the points with non-trivial stabilizer. In this case the non-trivial stabilizer has the dimension $1, 2$ or $3$  and  complexity of a singularity grows together with the dimension of its stabilizer. The space $X_5$ is no longer a manifold
 since  it is computed  in~\cite{BT-2}, see also~\cite{HS} that the   nontrivial homology groups for  $X_5$ are $\Z$ in dimensions $0$ and $8$, while it is $\Z _{2}$ in dimension $5$.

A point $x\in X_n$ is said to be a critical point  if its corresponding $T^n$-orbit in $G_{n,2}$   has a non-trivial stabilizer. As it  is proved in~\cite{BT-2},~\cite{BT-2nk} a  point $x\in X_n$ is a critical point if and only if  a   point from  its $T^n$-orbit is a singular  point of the smooth standard moment map $\mu _{n,2} : G_{n,2}\to \Delta _{n,2}$   in the sense of mathematical analysis, where $\Delta _{n,2}$ is considered as a smooth manifold with corners.  

On the other hand  for any $0 <k<n$ there  is the decomposition of  $G_{n,k}$   into the strata,  in which  a  stratum  $W_{\sigma}$   consists  of  those  $(\C ^{*})^{n}$ - orbits which have the same moment map image,  that is an interior of the  polytope  $P_{\sigma}\subset \Delta _{n,k}$. A polytope $P_{\sigma}$  we call   an admissible polytope for $W_{\sigma}$.    The  orbit space $F_{\sigma} = W_{\sigma}/(\C ^{*})^{n}$ is called a space of parameters of a  stratum $W_{\sigma}$ and it holds  $W_{\sigma}/T^n \cong \stackrel{\circ}{P}_{\sigma}\times F_{\sigma}$.  In particular the main stratum $W$ is a stratum whose  admissible polytope is the whole hypersimplex $\Delta _{n,k}$ and its space of parameters we denote by $F_{n,k}$.    The main  stratum is  an open dense set in $G_{n,k}$ which implies  that  the space $\stackrel{\circ}{\Delta}_{n,k}\times F_{n,k}$ is an open dense set in the orbit space $G_{n,k}/T^n$.  The $T^n$ -  action on $G_{n,k}$ defines    the  semi-upper continuous function from $G_{n,k}$ to the set  $S(T^n)$ of all toral subgroups in $T^n$, which to any point $L\in G_{n,k}$ assigns its stabilizer $T_{L}$. We proved in~\cite{BT-2} that this function is constant  $T_{\sigma}$ on each stratum $W_{\sigma}$ which implies that the torus  $T^{\sigma} = T^{n}/T_{\sigma}$ acts freely on $W_{\sigma}$.   The notions of the strata, corresponding  admissible polytopes and spaces of parameters have already been present  in some forms known in the literature. As we realized these notions are not enough for the description of an orbit space $G_{n,k}/T^n$. Therefore,   in our papers~\cite{BT-2},~\cite{BT-2nk} we introduced the new notions of universal space of parameters $\mathcal{F}_{n, k}$ and virtual spaces of parameters $\tilde{F}_{\sigma}\subset \mathcal{F}_{n,k}$ for the strata $W_{\sigma}$. The key properties of these notions are   that $\mathcal{F}_{n,k}$ is a compactification of the space of parameters  $F_{n,k}$ of the main stratum  and that  $\mathcal{F}_{n,k} = \cup _{\sigma}\tilde{F}_{\sigma}$. Moreover, there exists  the projection $p_{\sigma} : \tilde{F}_{\sigma} \to F_{\sigma}$ for any $\sigma$. We explicitly describe $\mathcal{F}_{n} = \mathcal{F}_{n,2}$ in~\cite{BT-U} and show that it   is a smooth compact manifold which can be identified with   the Chow quotient $G_{n,2}\!/\!/T^n$ defined by Kapranov~\cite{Kap}.


We define a  point $x\in G_{n,2}/T^n$ to be a singular point if   for the stratum $W_{\sigma}$ such that $x\in W_{\sigma}/T^n$,  the space of parameters $F_{\sigma}$ is not homeomorphic to the virtual space of parameters $\tilde{F}_{\sigma}$. It will be shown in the paper that this condition can be formulated in terms of  Pl\"ucker coordinates as well.

For $n\geq 5$ all critical points of $X_n$ belong to the set of singular points, while the case $n=4$ is somewhat special. 

\subsection{The main results}

The main goal  of this paper is to construct a smooth manifold  with corners $U_n$ which functorially  resolves the singular points  of an  orbit space $X_n$  being at the same time compatible with the  combinatorial structure of $X_n$ described below.  Let $Y_n = X_n\setminus \text{Sing}X_n$ and note that $Y_n$ is an open, dense set in $X_n$ which is a manifold.  By the resolution of singularities we  mean that  $U_n$ is a such manifold for which there exists  a  projection $p: U_n\to X_n$   such that for an open, dense submanifold  $V_n = p^{-1}(Y_n)\subset U_n$ the map $p : V_n\to Y_n$ is a diffeomorphism.

Note that the sequences of embeddings of the spaces $\C^2 \subset C^3 \subset \ldots \subset C^n \subset C^{n+1} \subset \ldots$ and the  groups  
$ T^2 \subset T^3\subset \ldots \subset T^n \subset T^{n+1}  \subset  \ldots$ define the sequence of embeddings of the Grassmann manifolds 
$G_{2,2}\subset G_{3,2}\subset \ldots G_{n,2} \subset G_{n+1,2}\subset \ldots$ and their orbit spaces 
$X_2 \subset X_3 \subset X_n \subset X_{n+1}\subset \ldots$.  We show that our construction has functorial property meaning that it  produces   the sequence of embeddings of the  smooth manifolds with corners which accordingly resolve the singular points of the orbit spaces $X_n$.

In the case of the toric manifolds the interiors of admissible polytopes do not intersect. The difficulties  in the description of the orbit spaces $G_{n,k}/T^n$ is caused by the fact that     there are, as a rule,  admissible polytopes whose interiors have non-empty intersection 

In the    paper of Goresky-MacPherson~\cite{GM}  it was  suggested the decomposition of $\Delta _{n,k}$ into disjoint union  of  chambers $C_{\omega}$. The chambers are obtained by  the intersections of the interiors of all admissible polytopes, that is $C_{\omega} = \cap _{\sigma \in \omega} \stackrel{\circ}{P}_{\sigma}$, such that $C_{\omega} \cap  \stackrel{\circ}{P}_{\sigma}= \emptyset$ if  $\sigma \not \in \omega$.  They pointed  that for any $x, y\in C_{\omega}$ the preimages  orbit spaces  $\mu ^{-1}(x)/T^n$ and $\mu ^{-1}(y)/T^n$ are homeomorphic, that is homeomorphic to some compact topological space $F_{\omega}$.   Note that $F_{\omega}$ has the  canonical decomposition as $F_{\omega} = \cup _{\sigma \in \omega} F_{\sigma}$. The difficulties with the  description of the orbit space $G_{n,k}/T^n$ - structure is caused   by the fact that  preimages  $F_{\omega _{1}}$ and $F_{\omega _{2}}$  are not in  general  homeomorphic for $\omega _1\neq \omega _2$ and one needs to make them correspondent.

We resolve this problem for the Grassmannians $G_{n,2}$ by showing that there exist a smooth manifold $\mathcal{F}_{n}$, that is a universal space of parameters,  and  the continuous  projection $G:  U_n =\Delta _{n,2}\times \mathcal{F}_{n} \to X_n$, see Theorem~\ref{prefin} and Theorem~\ref{fin}.  The projection $G$ defines on $U_n=\Delta _{n,2}\times \mathcal{F}_{n}$ an equivalence relation  which we explicitly describe in terms of the chamber decomposition  of $\Delta _{n,2}$ and the corresponding decompositions of  the manifold $\mathcal{F}_{n}$.  In the result we obtain that   the orbit space $X_n$ is a quotient space of a smooth manifold with corners $ U_n=\Delta _{n,2}\times \mathcal{F}_{n}$ by this equivalence relation. Our construction is based on two main inputs: {\bf first} we show that  the decomposition of $\Delta _{n,2}$ indicated by Goresky-MacPherson can be described in terms  of some special hyperplane arrangement;  {\bf second} we prove that for any point $x\in \Delta _{n,2}$  the union $\tilde{F}_{x} = \cup \tilde{F}_{\sigma}$ of all virtual spaces of parameters $\tilde{F}_{\sigma}$,  which correspond to the admissible polytopes $P_{\sigma}$ such that $x\in \stackrel{\circ}{P}_{\sigma}$ coincide with the universal space of parameters, that is   $\tilde{F}_{x}= \mathcal{F}_{n}$. 
 In the result we obtain the required  projection $U_n=\Delta _{n,2}\times \mathcal{F}_{n} \to X_n$. The smooth manifold with corners $U_n$  is a resolution of singularities of $X_n$ which is functorially  related  to the sequence of the  natural embeddings of these orbit spaces and which is compatible with the structure of the above chamber decomposition of the hypersimplex $\Delta _{n,2}$. 

Moreover,  it is defined  the action of the symmetric group $S_n$ on $X_n$  by   the action of the normalizer $N(T^n)$ in $U(n)$ on $G_{n,2} = U(n)/(U(2)\times U(n-2))$.  The group $S_n$ acts as well on $\Delta _{n,2}$ by the permutation of the vertices, this action preserves the chamber decomposition of $\Delta _{n,2}$  and  the induced moment map $\hat {\mu}_{n,,2} : X_n\to \Delta _{n,2}$ is $S_n$ - equivariant.  Using this one can define $S_n$-action on $U_n$  and the map  $G : U_n\to X_n$, which resolves the singular points of $X_n$ is $S_n$-equivariant.

We want to point that our results may have application  in the  study of positive complexity torus actions which has been recently extensively developing, see for example~\cite{AA-1},~\cite{AC},~\cite{K-T-1},~\cite{tim}. In that context the theory of spherical manifolds  including theory of homogeneous spaces of compact Lie groups provides numerous  examples of  positive complexity torus actions,~\cite{AC}.  The interesting examples arise as well in the theory of symplectic manifolds with
Hamiltonian torus actions, ~\cite{K-T}. In the recent time many non-trivial results about positive complexity torus actions are obtained by the methods of equivariant algebraic topology.

In addition, in the focus of recent studies are the description of  torus  orbit closures  for different varieties, see~\cite{ML},~\cite{MLP},~\cite{NO}. Our geometric analytical description of the admissible, that is matroid, polytopes for  $G_{n,2}$ may lead to the results in this direction.

\section{$T^n$- equivariant automorphisms of $G_{n,k}$}
We first analyse the relations between the $T^n$-equivariant automorphisms of $G_{n,k}$ and the standard moment map $\mu _{n,k} : G_{n,k}\to \R ^n$. 
  
Let $G_{n,k}$ be a complex Grassmann manifold of $k$-dimensional complex planes in $\C ^{n}$. It is known~\cite{Chow} that the group of holomorphic automorphisms $\text{Aut} G_{n,k}$  is isomorphic to the projective group $PU(n)$ if $n\neq 2k$, while  for $G_{2k,k}$ it is isomorphic to $\Z _{2}\times PU(n)$, where $n=2k$.  The group $\Z _{2}$ is defined by the duality isomorphism.  The duality isomorphism for $G_{2k,k}$ we obtain from  the  standard  diffeomorphism $c_{n,k} : G_{n, k}\to G_{n, n-k}$: denote by $l : \C ^{n} \to (\C ^{n})^{*}$  the canonical isomorphism, where $(\C ^{n})^{*}$ is the dual space for $\C ^{n}$,  let  $L\in G_{n,k}$ and put $L^{'}=\{ \lambda \in (\mathbb C^{*})^{n} | \lambda (v)=0\; \text{for any}\; v\in L\}$. Then it is defined $c_{n,k}(L) = l^{-1}(L^{'})$ and for $n=2k$ it gives one more   non-linear isomorphism for $G_{2k,k}$.

Further,  the symmetric group $S_n$  acts on $\C ^n$ by permuting the coordinates, so  any $\mathfrak{s} \in S_n$ produces the automorphism $\mathfrak{s} : G_{n,k}\to G_{n,k}$. It follows that  $S_{n}$ is a subgroup of $\text{Aut}G_{n,k}$.   

 Let us consider the canonical action of the compact torus $T^{n}$ on $G_{n,k}$ and let $S^{1}$ be the diagonal circle in $T^n$. Then $T^{n-1} = T^n/S^1$ acts effectively on $G_{n,k}$ and it is a maximal torus in $PU(n)$.  The normalizer 
of $T^{n-1}$ in $PU(n)$ is 
$T^{n-1}\rtimes S_{n}$. The canonical diffeomorphism $c_{n,k} : G_{n,k}\to G_{n, n-k}$ is  equivariant for the canonical torus action as well. We summarize:

\begin{lem}
 The  subgroup  of $\text{Aut} G_{n,k}$ which contains those  elements that commutes with the canonical $T^n$-action on $G_{n,k}$ is $T^{n-1}\rtimes S_n$ for $n\neq 2k$ and it is $\Z_{2}\times (T^{n-1}\rtimes S_n)$ for $n=2k$.  
\end{lem}


Let $W_{0}\subset G_{n,k}/T^n$ be the  set of fixed point for  the considered  $T^n$-action on $G_{n,k}$. We find useful to note the following:

\begin{prop}\label{sub}
Let $H$ be a subgroup of those holomorphic automorphisms  for  $G_{n,k}$  for which the set $W_{0}$ is invariant. 
Then  $H= T^{n-1}\rtimes S_n$ for $n\neq 2k$ and i $H=\Z_{2}\times (T^{n-1}\rtimes S_n)$ for $n=2k$.  
\end{prop}
\begin{proof}
We need to prove that if $f\in H$ then $f$ commutes with the canonical $T^n$-action on $G_{n,k}$.  
The fixed points for the canonical action of $T^{n}$ on $G_{n,k}$ are $k$-dimensional coordinate subspaces in $\C ^{n}$. Let $v$ be a coordinate vector. There are ${n-1\choose k-1}$ coordinates $k$-subspaces which contains $v$. Let $L$  be a such  subspace and  $f\in PU(n)$  which leaves $W_{0}$ invariant. For the linear isomorphism $\hat{f}$ representing  $f$ it holds that $\hat{f}(L)$ is a $k$-dimensional coordinate subspace defined by $z_{i_1}=\ldots = z_{i_k}=0$ for some $1\leq i_1<\ldots <i_k\leq n$,   which implies that the coordinates of $\hat{f}(v)$ indexed by $i_1,\ldots ,i_k$ are equal to zero.  Since there are at least $n-1$ different $k$-dimensional  coordinate subspaces which contains $v$  it follows that $\hat{f}(v)$ has to be, up to constant,  a coordinate vector. Therefore the isomorphism  $\hat{f}$, up to constants,  permutes the coordinate vectors. It implies that $\hat{f}$ commutes with $T^n$ - action on $\C ^{n}$, so $f$ commutes with the induced  $T^n$-action on $G_{n,k}$.
When $k=2n$, the nontrivial map $f\in \Z _{2}$ any $k$-dimensional coordinate subspace maps to its orthogonal complement which is a $k$-coordinate subspace as well.
\end{proof}

 Now, let $\mu _{n,k} : G_{n,k}\to  \Delta _{n,k}\subset \R ^{n}$ be the standard moment map that is
\[
\mu _{n,k}(L) = \frac{1}{\sum |P^{J}(L)|^{2}}\sum |P^{J}(L)|^2 \Lambda _{J},
\]
where $\Lambda _{J}\in \R^n$, $\Lambda _{i}=1$ for $i\in J$, while $\Lambda _{i}=0$ for $i\notin J$ and $J\subset \{1,\ldots ,n\}$, $\|J\|=k$.


Assume  that $f\in \text{Aut}G_{n,k}$   is  an automorphism   for which  there exists a combinatorial  isomorphism $\bar{f} : \Delta _{n,k} \to \Delta _{n,k}$ such that the following diagram commutes:
\begin{equation}\label{diag}
\begin{CD}
 G_{n,k} @>{f}>>G_{n,k}\\
@VV \mu _{n,k} V @VV\mu _{n,k} V \\ 
 \Delta _{n,k}@>{\bar{f}}>>\Delta _{n,k},
\end{CD}\end{equation}

that is $\mu _{n,k} \circ f= \bar{f}\circ \mu _{n,k}$.

\begin{lem}
If an automorphism   $f$  of $G_{n,k}$ satisfies the assumption~\eqref{diag}  then $f\in T^{n-1}\rtimes  S_{n}$ for $n\neq 2k$ and $f\in T^{n-1}\times (\Z _{2}\times S_{n})$ for $n=2k$.
\end{lem}
\begin{proof}
Since $\mu _{n,k} $ gives a bijection between the set $W_{0}$ of  fixed points for $T^n$-action on $G_{n,k}$ and the vertices of $\Delta _{n,k}$, and, by assumption $\bar{f}$ is a combinatorial isomorphism, it follows that $f$ leaves the set $W_{0}$ invariant. The statement follows from  Proposition~\ref{sub}.
\end{proof}

We prove now that the automorphisms which satisfy~\eqref{diag} in fact  coincide with $T^{n-1}\rtimes S_n$  or $\Z_{2}\times (T^{n-1}\rtimes S_n)$. 

Since $\mu _{n,k}$ is $T^n$-invariant it follows that any $f \in T^{n-1} < \text{Aut}G_{n,k}$ satisfies~\eqref{diag} as one can take $\bar{f}=id_{\Delta _{n,k}}$.  The action of symmetric group  $S_{n}$  on $G_{n,k}$ induces  the  action of $S_n$ on $\Delta_{n,k}$ given by the permutations of coordinates.  

\begin{lem}
For any $\mathfrak{s}\in  S_n$ the combinatorial automorphism $\bar{\mathfrak{s}} : \Delta _{n,k}\to \Delta _{n,k}$ such the diagram~\eqref{diag} commutes  is  given by the  permutation of coordinates  in $\R ^n$ which is  defined  by $\mathfrak{s}$. 
\end{lem}
\begin{proof}
Since  $P^{J}(\mathfrak{s} (L))= P^{\mathfrak{s} (J)}(L)= P^{J}(L)$ for any $L\in G_{n,k}$  it follows that $\mu _{n,k} (\mathfrak{s} (L))= ((pr _{1}\circ \mu _{n,k}) (\mathfrak{s} (L)),\ldots ,(pr _{n}\circ \mu _{n,k})(\mathfrak{s} (L)) = ((pr_{\mathfrak{s}(1)}\circ\mu _{n,k})(L), \ldots , (pr_{\mathfrak{s} (n)}\circ \mu _{n,k})(L))= \bar{\mathfrak{s}} (\mu _{n,k}(L))$, where $pr _{i} : \R ^{n}\to \R ^{1}_{i}$ is the projection on $i$-th coordinate, $1\leq i\leq n$.
\end{proof}

As for the involutive automorphism of $G_{2k,k}$ we first  recall an explicit description  of the diffeomorphism  $c_{n,k} : G_{n,k}\to G_{n, n-k}$.  Let $L\in G_{n, k}$ and assume that the Pl\"ucker coordinate $P^{J}(L)\neq 0$, where $J=\{1,\ldots, k\}$. Then there exists the  base for $L$ in which $L$  can be represented by $n\times k$-matrix whose submatrix which consists of the first $k$ columns is an identity matrix, that is
\begin{equation}
L=\left(
\begin{array}{cc}
I\\
A
\end{array}
\right),
\; \text{where}\;  A\;\text{is}\;  (n-k)\times k\; \text{ matrix}.
\end{equation}
It is easy to check that  the $(n-k)$-dimensional subspace $l^{-1}(L^{'})$ can be represented by the matrix
\begin{equation}
l^{-1}(L^{'})=\left(
\begin{array}{cc}
-A^{T}\\
I
\end{array}
\right),
\; \text{where}\;  I\;\text{is}\;  (n-k)\times (n-k)\; \text{identity  matrix}.
\end{equation}
 
It immediately implies the following:
\begin{lem}\label{Pl}
For any $J\subset \{1,\ldots , n\}$, $\|J\|=k$ it holds
\begin{equation}\label{PLC}
P^{J}(L)= \pm P^{\bar{J}}(l^{-1}(L^{'})), \; \text{where}\; \bar{J} = \{1,\ldots ,n\}\setminus J.
\end{equation}
\end{lem}

Then Lemma~\ref{Pl} implies:
\begin{lem}\label{cnk}
There exists a combinatorial isomorphism $\bar{c}_{n,k} : \Delta _{n,k}\to \Delta _{n, n-k}$ such that the following  diagram commutes
\begin{equation*}\begin{CD}
 G_{n,k} @>{c_{n,k}}>>G_{n,n-k}\\
@VV {\mu _{n,k} }V @VV\mu _{n,n-k}V \\ 
 \Delta _{n,k}@>{\bar{c}_{n,k}}>>\Delta _{n,n-k}.
\end{CD}\end{equation*}

\end{lem}

\begin{proof}
Consider an isomorphism $\bar{c}_{n,k} : \Delta _{n,k}\to \Delta _{n,n-k}$ which sends  $x=\sum\limits _{J\subset \{1,\ldots, n\}, \|J\|=k} \alpha_{J}\Lambda _{J}$ to $\bar{c}_{n,k}(x)= \sum\limits _{J\subset \{1,\ldots, n\}, \|J\|=k} \alpha_{\bar {J}}\Lambda _{\bar{J}}$, where 
$\bar{J}= \{1, \ldots n\}\setminus J$. It directly follows from~\eqref{PLC} that $\bar{c}_{n,k}$ is a  required combinatorial isomorphism.
\end{proof}

For $n=2k$ we obtain the  automorphism $\bar{c}_{2k,k} : \Delta _{2k,k}\to \Delta _{2k,k}$ which is  given by
\begin{equation}\label{c2kk}
\bar{c}_{2k,k}(x) = \sum\limits _{J\subset \{1,\ldots ,2k\}, \|J\|=k} \alpha_{\bar {J}}({\bf 1} - \Lambda _{J}),
\end{equation}
where   $x=\sum\limits _{J\subset \{1,\ldots, n\}, \|J\|=k} \alpha_{J}\Lambda _{J}$ and ${\bf 1}=(1,\ldots ,1)$.

Together with Lemma~\ref{cnk} this implies:

\begin{lem}
 For $n=2k$ the isomorphism $\bar{c}_{2k,k} : \Delta _{2k,k}\to \Delta _{2k,k}$  is given by
\begin{equation}\label{form}
x \to {\bf 1}-x,
\end{equation}
where ${\bf 1}= (1,\ldots ,1)$.
\end{lem}
\begin{proof}

For $x=(x_1,\ldots, x_{2k})\in \Delta _{2k, k}$ we have that   
\[
x_i = \frac{1}{\sum\limits _{J, \|J\|=k}|P^{J}(L)|^2}\sum\limits_{J, i\in J, \|J\|=k}|P^{J}(L)|^2 \;  \text{for some}\;  L\in G_{n,k}.
\]
Then from~\eqref{c2kk} it follows that
\[
\bar{c}_{2k,k}(x)_{i} =  \frac{1}{\sum\limits _{J\subset \{1,\ldots ,n\}, \|J\|=k}|P^{J}(L)|^2}\sum\limits_{J\subset \{1,\ldots .n\}\setminus \{i\}, \ |J\|=k}|P^{J}(L)|^2 =
\]
\[
\frac{1}{k}(x_2+\ldots +x_n-(k-1)x_1, \ldots, x_1+\ldots +x_{n-1}-(k-1)x_n) = (1-x_1,\ldots, 1-x_n),
\]
since $x_1+\ldots x_n=k$. 
\end{proof}

Altogether we obtain:

\begin{prop}\label{ffinal}
An element  $f\in G_{n,k}$  satisfies the assumption~\eqref{diag}  if and only if  $f\in T^{n-1}\rtimes S_{n}$ for $n\neq 2k$ and $f\in \Z _{2}\times (T^{n-1}\rtimes S_{n})$ for $n=2k$.
\end{prop}

For the preimages of the points form $\Delta _{n,k}$ by the moment map it  holds:

\begin{lem}
 If $f\in \text{Aut}G_{n,k}$ satisfies  assumption~\eqref{diag} then    $\mu^{-1}_{n,k}(x)$ is homeomorphic to $\mu ^{-1}_{n,k}(\bar{f}(x))$ for any $x\in \Delta _{n,k}$.
\end{lem}
\begin{proof}
From~\eqref{diag} it directly follows that $f : \mu ^{-1}_{n,k}(x) \to \mu ^{-1}_{n,k}(\bar{f}(x))$. Since $f$ is an automorphism the statement follows.
\end{proof}

Since any automorphism for $G_{n,k}$ which commutes with $T^n$-action produces a homeomorphism of $G_{n,k}/T^n$ and the moment map is $T^n$-invariant,  we deduce from Proposition~\ref{ffinal} the following.

\begin{cor}\label{cormain}
For $n\neq 2k$,  the subspaces $\mu ^{-1}_{n,k}(x)/T^n, \mu ^{-1}_{n,k}(\mathfrak{s}(x))/T^n \subset G_{n,k}/T^n$ are homeomorphic for any $x\in \Delta _{n,k}$  and any  $\mathfrak{s} \in S_n$. For $n=2k$, the subspaces  $\mu ^{-1}_{n,k}(x)/T^n, \mu ^{-1}_{n,k}(\bar{f}(x))/T^n$  are homeomorphic for any $x\in \Delta_ {n,k}$ and any  $\bar{f}\in S_{n}$ or $\bar{f} = \bf{1} - x$.
\end{cor} 

\section{Grassmann manifolds  $G_{n,2}$}
\subsection{Admissible polytopes}
We first recall the notions of an  admissible polytope and a stratum as introduced in~\cite{BT-2}. Some other equivalent definition may be found in~\cite{GS}. Let $M_{ij}=\{L\in G_{n,2} | P^{ij}(L)\neq 0\}$, where $\{i,j\}\subset \{1,\ldots, n\}$, $i<j$ be the standard Pl\"ucker charts on $G_{n,2}$    and put $Y_{ij}= G_{n,k}\setminus M_{ij}$. A nonempty set 
\[
W_{\sigma} = \big(\bigcap\limits_{\{i,j\}\in \sigma}M_{ij}\big) \cap  \big(\bigcap\limits_{\{i,j\}\notin \sigma}Y_{ij}\big)
\]
is called a stratum, where $\sigma \subset {n\choose 2}$. For a stratum $W_{\sigma}$ we have  that $\mu (W_{\sigma}) = \stackrel{\circ}{P}_{\sigma}$, where $P_{\sigma}$ is a convex hull of the vertices $\Lambda _{ij}$, $\{i,j\}\in \sigma$. A polytope $P_{\sigma}$ which can be obtained in this way is said  to be an admissible polytope. In addition, all points from $W_{\sigma}$ has the same stabilizer $T_{\sigma}\subseteq  T^n$ and the torus $T^{\sigma} = T^n/T_{\sigma}$ acts freely on $W_{\sigma}$, see~\cite{BT-2}.

The boundary $\partial \Delta _{n,2}$  consists of $n$ copies of the hypersimplex $\Delta _{n-1,2}$ and $n$ copies of the simplex  $\Delta _{n-1,1}$. Moreover,  $\mu ^{-1}(\partial \Delta _{n,2}) = n\# \mu ^{-1}(\Delta _{n-1, 2}) \cup n\#\mu ^{-1}(\Delta _{n-1, 1}) = n\# G_{n-1, 2} \cup n\# \C P^{n-2}$ . It follows that the admissible polytopes for $G_{n,2}$ can be described inductively, by describing those which intersect the  interior of $\Delta _{n,2}$.   We have that $\dim \Delta _{n,2} = n-1$  and that  $\dim P_{\sigma}$ is equal to the dimension of the torus $T^{\sigma}$  which acts freely on the stratum $W_{\sigma}$ for any admissible polytope $P_{\sigma}$ in $\Delta _{n,2}$, see~\cite{BT-2}.

Note that one can find   in~\cite{Kap} as well the description  of  the admissible polytopes $P_{\sigma}\subset \Delta _{n,2}$  for $T^n$-action in $G_{n,2}$  in terms of matroid theory and Chow quotient $G_{n,2}\!/\!/ (\C ^{\ast})^{n}$:   there are first described all matroid decomposition of $\Delta _{n,2}$ and afterwards it is proved that there are all realizable, that is  all of them  come from some non-empty Chow strata in $G_{n,2}\!/\!/ (\C ^{\ast})^{n}$.

We provide  here the new purely analytical description  of the admissible polytopes for $G_{n,2}$,  which is going to be suitable for our purpose of the description of an  orbit space $G_{n,2}/T^n$.

We  find  useful to point to the following observation. Let  $pr _{i} : \R ^{n}\to \R _{i}^{1}$ denotes the  projection on $i$-th coordinate, $1\leq i\leq n$. 
\begin{lem}\label{polytopbound}
A point $x\in \Delta _{n,k}$ belongs to the boundary $\partial \Delta _{n,k}$ if and only if $pr _{i}(x) =0$ or $pr _{i}(x)=1$ for some $1\leq i\leq n$. Moreover, a polytope $P$ which is the convex span of some vertices of $\Delta _{n,k}$ belongs to $\partial \Delta _{n,k}$ if and only if there exists $1\leq i\leq n$ such that  $pr _{i}(x)=0$ for all $x\in P$ or $pr _{i}(x)=1$ for all $x\in P$.  
\end{lem}

\begin{prop}
Any admissible polytope for $G_{n,2}$ whose dimension is $\leq n-3$ belongs to the boundary $\partial \Delta _{n,2}$ for $\Delta _{n,2}$.
\end{prop}

\begin{proof}
Let $P_{\sigma}$ be an admissible polytope and   let $\dim P_{\sigma}= q\leq n-3$.   Because of the action of the symmetric group, we can assume that the $V_{0} =(1,1,0,\ldots, 0)$ is a vertex for $P_{\sigma}$. Then  there exist $q$ vertices $V_1,\ldots , V_{q}$ for $P$ which are adjacent to $V_{0}$ and the dimension of the subspace $G$ spanned by the vectors  $V_0-V_1, \ldots ,V_{0}-V_{q}$ is equal  to $q$. The polytope $P_{\sigma}$ is  the intersection of hypersimplex $\Delta _{n,2}$ and the $q$-dimensional  plane which contains the point $V_{0}$ and which is  directed  by  the subspace $G$. Since $V_1, \ldots V_q$ are adjacent  to $V_0$ they must have  $0$ and $1$ at the first two coordinate  places. Thus, any  of $V_{i}$'s, $1\leq i\leq q$  has exactly one $1$ and the last $n-2$ places. Since 
$q\leq n-3$, all $V_{i}$, $1\leq i\leq q$  must have one common coordinate $x_j$, $j\geq 3$ which is equal to zero. It implies that $P_{\sigma}$ belongs to the boundary $\partial \Delta _{n,2}$.
\end{proof}

\subsection{Admissible polytopes of dimension $n-2$}\label{adm-n-2}

The admissible polytopes of   dimension $n-2$ can be divided into those which  belong to the boundary $\partial \Delta _{n, 2}$ and those which intersect the interior of $\Delta _{n,2}$.  Those on $\partial \Delta _{n,2}$ are given by $n$ hypersimplices $\Delta _{n-1,2}$ and their $(n-2)$-dimensional admissible polytopes, and $\Delta _{n-1,1}$ which is the simplex $\Delta ^{n-2}$ .

We describe here those admissible $(n-2)$-dimensional polytopes which  intersect the interior of $\Delta _{n,2}$.

Denote by $\Pi _{\{i,j\}}$  the set of all  $(n-2)$-dimensional planes  which contain the vertex  $\Lambda _{\{i,j\}}=e_i+e_j$,  which are  parallel to  the  edges of $\Delta _{n,2}$ that are adjacent  to  $\Lambda_{\{i,j\}}$, and which intersect the interior of $\Delta _{n,2}$.  The edges  adjacent $\Lambda_{\{i,j\}}$   are given by
\[
e_{\{j,s\}} = \Lambda _{\{i,j\}} - \Lambda _{\{i,s\}},\; s\neq i,j, 
\]
\[
e_{\{i,q\}} = \Lambda_{\{i,j\}}-\Lambda _{\{q,j\}}, \; q\neq i,j.
\]
Note that the vectors $e_{\{j,m\}}$ and $e_{\{i,s\}}$ represent the complementary roots for $U(2)\times U(n-2)$ related to $U(n)$ and $G_{n,2}$ can be represented as the homogeneous space $U(n)/U(2)\times U(n-2)$.
 The planes $\Pi _{\{i,j\}}$ can be described  as follows:
\begin{lem}\label{planes}
The set $\Pi _{\{i,j\}}$ consists of the planes 
\[
\alpha _{\{i,j\}, l}^{S} = \Lambda _{\{i,j\}} + F_{l, S},\;\; 1\leq l\leq n-3,\; S\subset \{1, \ldots , n\}, \; \| S\| =l 
\]
and  $i, j\notin S$. The directrix $F_{l, S}$  is  spanned   by the vectors 
\[
e_{\{j, s\}}\;  \text{and}\; e_{\{i,q\}},\;\;  s\in S, \; \;  q\notin S \cup \{i,j\}.
\]

\end{lem}

It is easy to verify that the planes  $\alpha _{\{i, \j\}, l}^{S}$ can be more explicitly written:

\begin{cor}\label{eq1}
The set $\Pi _{\{i,j\}}$ consists of the $(n-2)$-dimensional planes which are obtained as the intersection of the plane $\sum\limits_{i=1}^{n}x_{i}=2$ with the planes  
\begin{equation}
x_{j}+\sum\limits _{s\in S}x_s = 1,
\end{equation}
where $S \subset \{1, \ldots , n\}$, $\|S\| = l$, $1\leq l\leq n-3$  and $i, j\notin S$. These planes  coincide with the planes
\begin{equation}
x_{i}+\sum\limits_{s\notin \{ i,j\}\cup S}x_{s}=1.
\end{equation}  
\end{cor}

The group $S_{n}$ acts on the set $\{\Pi _{\{i,j\}}, 1\leq i<j\leq n\}$. The stabilizer of $S_{n}$-action   on the set $\{\Pi _{\{i,j\}}, 1\leq i <j\leq n\}$ is the group  $S_{2}\times S_{n-2}$, that is  the  group $S_{2}\times S_{n-2}$ acts on  the set $\Pi _{\{i,j\}}$.


 We  first prove the following:
\begin{prop}\label{main}
The admissible polytopes of dimension $n-2$ which do not belong to the boundary $\partial \Delta _{n,2}$  coincide with the polytopes obtained by the intersection of $\Delta _{n,2}$ with  the planes from $\Pi _{\{i,j\}}$, where $1\leq i<j\leq n$.
\end{prop}  

\begin{proof}
Let $P_{\sigma}$ be an admissible polytope which does not belong to the boundary $\partial \Delta _{n,2}$. Because of the action of the symmetric group $S_n$, we can assume that $\Lambda_{\{1,2\}}$ is a vertex of $P_{\sigma}$. Since $\dim P_{\sigma}=n-2$ it follows that $\Lambda _{\{1,2\}}$ has $n-2$ linearly  independent adjacent  vertices in $P_{\sigma}$. Any of these  vertices has one $1$ and one $0$ at the  first two coordinate places. Let $l$ be the number of these vertices having $1$ as the first coordinate, then    they have  exactly one $1$ at the last $n-2$ coordinates.   Due to the $S_n$-action we can assume that  these vertices   are $\Lambda _{\{1,3\}}, \ldots, \Lambda _{\{1,l+2\}}$.  The other $n-2-l$ vertices   have $0$ as the first coordinate,  $1$ as the second  coordinate and they all have exactly one $1$ at the last $n-2$ coordinates.  We claim that $1$ has to be among the last $n-l-2$ coordinates  for  all these $n-2-l$ vertices. If this is not the case,  one of these vertices  would have the last $n-2-l$ coordinates all  equal to zero.  For the other of these  vertices    we eventually would have $1$ at the last   $n-l-2$ coordinate places, 
and these places have to be different for different vertices.
The number  of such  vertices  is  at most $n-3-l$, so they  must have   one common of the last $n-l-2$ coordinates which is   equal to  zero  and consequently it is zero coordinate for all $n-2-l$ vertices, that is for all $n-2$ vertices.   It would imply that the polytope $P_{\sigma}$ belongs to $\partial \Delta _{n,2}$.

Therefore,  the other $n-2-l$ vertices  are $\Lambda _{\{2,j\}}$, $l+2\leq j\leq n$. It follows that the polytope $P_{\sigma}$,  up to the action of the symmetric group $S_{n}$ belongs to the plane $\alpha _{\{1,2\} ,l} = \alpha _{\{1,2\}, l}^{S}$ where $S=\{3, \ldots ,l+2\}$ for some $1\leq l\leq n-3$. 

Let now $P$ be a polytope, which is obtained  as the intersection of $\Delta _{n,2}$ with a plane from the set $\Pi _{\{1,2\}}$, up to the action of the group $S_{n}$.  The points  of the plane  $\alpha _{\{1,2\}, l}$ can be explicitly written in  $\R ^{n}$ as:
\[
(1+a_{l+1}+\ldots +a_{n-2}, 1+a_1+\ldots +a_{l}, -a_1, \ldots ,-a_{n-2}), \;\:a_{i}\in \R, \; 1\leq i\leq n-2.
\]
It follows that the vertices of $\Delta _{n,2}$ which belong to this plane are $\Lambda _{\{1,j\}}$,  $2\leq j\leq l+2$, $\Lambda _{\{2,\}j}$, $l+3\leq j\leq n$ and  $\Lambda _{\{i,j\}}$, $3\leq i\leq l+2$, $l+3\leq j\leq n$, that is $P$ is the  convex hull of these vertices.
If we consider the point $L\in G_{n,2}$ given by the matrix $A_{L}$ such that $a_{11}=a_{22}=1$, $a_{12}=a_{21}=0$,  
$a_{i1}= a_{2j}=0$, $3\leq i\leq l+2$, $l+3\leq j\leq n$ and $a_{1j}\neq 0$, $l+3\leq j \leq n$, $a_{2j}\neq 0$, $3\leq j\leq l+2$,  we see that 
the image by the moment map of the closure $(\C ^{*})^{n}$-orbit of the point $L$ is the polytope $P$. Thus, $P$ is an admissible polytope.   
\end{proof}

\begin{prop}
The number  of irreducible representations for $S_{2}\times S_{n-2}$-action on $\Pi _{\{i,j\}}$ is $[\frac{n-2}{2}]$. The  dimensions  of these irreducible representations are:
\[ 
\text{for}\; n\; \text{odd} :\;  {n-2 \choose l},\; 1\leq l\leq [\frac{n-2}{2}]\; ,
\]
 \[ 
\text{for}\;  n\;  \text{even} :\; {n-2 \choose l},\; 1\leq l< [\frac{n-2}{2}]\; 
 \text{and} \; \frac{2}{n-2}{n-2 \choose \frac{n-2}{2}}
\]
\end{prop}

\begin{proof}
 The number of planes in the set  $\Pi _{\{i,j\}}$ is 
\[
|\Pi _{\{i,j\}}| = \sum _{l=1}^{n-3}{n-2\choose l} =2^{n-2}-2.
\]
It implies that the group $S_{n-2}$ acts on the set
 $\Gamma _{\{i,j\}} = \Pi_{\{i,j\}}/S_{2}$ which consists of   $q=2^{n-3}-1$ elements. Moreover,   from  the description of the planes from $\Pi _{\{i,j\}}$ it follows that the  set of generators for $S_{n-2}$- action on $\Gamma _{\{i,j\}}$  is given by the planes $\alpha _{\{i,j\},l} = \alpha _{\{i, j\}, l}^{S}$, where $S=\{1, \ldots , l+\delta _{\{i,j\}}\}$, $\delta_{\{i, j\}} = 0, 1,2 $ corresponding to the cases $l<i<j$, $i\leq l< j$, $i<j\leq l$  and $1\leq l\leq [\frac{n-2}{2}]$.
The stabilizer of the element $\alpha _{\{i,j\}, l}$ is $S_{l}\times S_{n-2-l}$ for $1\leq l<[\frac{n-2}{2}]$. For  $l=[\frac{n-2}{2}]$ and $n$ odd, that is $l=\frac{n-3}{2}$, the stabilizer is  $S_{l}\times S_{n-2-l}$ , while for $n$ even, that is $l=\frac{n-2}{2}$,  the stabilizer is $S_{l}\times S_{l}\times S_{l}$. It follows that the $S_{n-2}$ action on $\Gamma _{\{i,j\}}$ corresponds to the representation of $S_{n-2}$ in $\C ^{2^{n-3}-1}$ whose irreducible summands for $n$ odd are in dimensions $\frac{(n-2)!}{l!(n-2-l)!}$, $1\leq l\leq [\frac{n-2}{2}]$, while for $n$ even they are in dimensions  $\frac{(n-2)!}{l!(n-2-l)!}$, $1\leq l< [\frac{n-2}{2}]$
 and $\frac{(n-2)!}{(\frac{n-2}{2}!)^{3}}$.
\end{proof}

In this way, the Proposition~\ref{main} can be improved as follows:

\begin{cor}\label{final}
The admissible polytopes of dimension $n-2$ which do not belong to the boundary $\partial \Delta _{n,2}$  are , up to the action of the symmetric group $S_{n}$,  given by the intersection of $\Delta _{n,2}$ with  the planes $\alpha _{\{1,2\}, l}$, where $1\leq l\leq [\frac{n-2}{2}]$.
\end{cor}  

Using Corollary~\ref{eq1} we can summarize the previous result as follows:

\begin{thm}\label{n-2}
The admissible polytopes for $G_{n,2}$ of dimension $n-2$ which have non-empty intersection with $\stackrel{\circ}{\Delta}_{n,2}$ are given by the intersection  of $\Delta _{n,2}$ with  the set $\Pi$ which consists of the  planes:
\begin{equation}\label{novo}
 \sum\limits_{i\in S, \|S\|=p}x_{i} =1, \;\; \text{where}\;\; S\subset \{1, \ldots ,n\}, \;  2\leq p\leq [\frac{n}{2}].
\end{equation}
\end{thm}


We also find useful to  note the following:
\begin{cor}
An admissible polytope of dimension $n-2$ defined by the hyperplane $x_k+ \sum\limits_{i\in S, k\notin S}
x_{i}=1$,  is a convex span of the vertices $\Lambda _{i, k}$  for $i\in S$,   where $S\subset \{1,\ldots , n\}$, $\|S\|=p$,  $2\leq p\leq [\frac{n}{2}]$ and $1\leq k\leq n$.
\end{cor}

We  describe as well the numbers of vertices of these polytopes:

\begin{cor}
An admissible polytope of dimension $n-2$  which does not belong to the boundary of $\Delta _{n,2}$ has  
\[
n_{p} = p(n-p), 
\] 
vertices for some $2\leq p\leq [\frac{n}{2}]$. 

Moreover, the number $q_{p}$  of admissible  polytopes which have $n_{p}$,  $2\leq p\leq [\frac{n}{2}]$  vertices is 
\[
q_{p} = {n\choose p}\;\;  \text{for}\; n\; \text{odd},
\] 
\[
q_{p} = {n\choose p} \;\; \text{for}\; n\; \text{even and}\; 1\leq l< [\frac{n-2}{2}],
\]
\[
q_{\frac{n}{2}} = \frac{{n\choose \frac{n}{2}}}{2} \;\; \text{for}\; n\; \text{even}.
\] 
\end{cor}

\begin{ex}
It follows from Corollary~\ref{final} that $G_{4,2}$   has one generating  admissible polytope  of  dimension   $2$ which is inside $\Delta _{4,2}$. This  polytope  has $4$ vertices and its $S_4$-orbit consists of  
$3$  polytopes, which give all admissible polytopes of dimension $2$ that  are  in $\stackrel{\circ}{\Delta} _{4,2}$. These polytopes are defined by the planes $x_1+x_2=1$, $x_1+x_3=1$ and $x_1+x_4=1$.
\end{ex}
\begin{ex}
 The Grassmannian  $G_{5,2}$ has   one generating admissible interior polytope of dimension  $3$ which is  inside $\Delta_{5,2}$.  It  has $6$ vertices and its $S_{5}$-orbit has $10$ elements  which  give all  admissible polytopes of dimension $3$ that are in 
$\stackrel{\circ}{\Delta} _{5,2}$. These polytopes are defined by the planes $x_{i}+x_{j}=1$ where $1\leq i<j\leq 5$.
\end{ex}
\begin{ex}
 The Grassmannian   $G_{6,2}$ has  $2$ generating  admissible polytopes of dimension  $4$ inside $\Delta _{6,2}$. These polytopes have  $8$ and $9$ vertices and their $S_6$-orbits  consists of $15$ and $10$ elements respectively.  They are defined by the planes $x_i+x_j=1$, $1\leq i<j\leq 6$ and $x_1+x_j+x_k=1$, $2\leq i<j\leq 6$.  Note that here  the corresponding representation for  $S_{2}\times S_{4}$- action on $\C ^{7}$ has $2$ irreducible summands of dimension $4$ and $3$. 
\end{ex}

\subsection{Admissible polytopes of dimension $n-1$}
Before to describe the admissible polytopes of dimension $n-1$, we  first need  the following result.

\begin{lem}\label{empty}
Assume that the  points of an  admissible polytope $P_{\sigma}$, $\dim P_{\sigma}=n-1$  satisfy inequalities
\[
\sum\limits _{i\in I} x_{i}\leq 1\; \text{and}\; \sum\limits _{j\in J}x_{j}\leq 1,
\]
where $I, J\subset \{1, \ldots, n\}$.
If $I \cap J \neq \emptyset$  then the points of $P_{\sigma}$ satisfy as well
 an inequality
\[
\sum\limits _{s\in I\cup J}x_{s}\leq 1.
\]
\end{lem}
 \begin{proof}
If $I\cap J\neq \emptyset$  then $P_{\sigma}$ does not contain the vertices $\Lambda _{ij}$, $i\in I$, $j\in J$, that is $P^{ij}(L)=0$ for all points $L$ from the stratum $W_{\sigma}$.  Now if $x=(x_1, \ldots , x_n)\in P_{\sigma}$ then $x=\mu (L)$ for some $L\in W_{\sigma}$. Note that  for $s\in I\cup J$ we have 
\[
x_{s} = \frac{S_{s}}{S},\;\;  S_{s}=\;\; \sum\limits_{m\notin I \cup J} P^{sm}(L),
\]
and  $S=\sum\limits _{1\leq p<q\leq n}P^{pq}(L)$.
It follows that in the different sums $S_{s}$, $s\in I\cup J$ contribute different Pl\"ucker coordinates $P^{sm}(L)$.  Therefore
\[
\sum\limits _{s\in I\cup J}x_s = \frac{1}{S}\sum\limits _{s\in I\cup J}S_{s} \leq 1.
\]
\end{proof}

\begin{cor}
If $\sum\limits_{i\in I}x_{i}=1$ and $\sum\limits_{j\in J}x_{j}=1$ are the facets of an admissible polytope $P_{\sigma}$, $\dim P_{\sigma}=n-1$, $\|I\|, \|J\|\geq 2$  and the points of $P_{\sigma}$ satisfy $\sum\limits_{i\in I}x_{i}\leq 1$ and $\sum\limits_{j\in J}x_{j}\leq 1$ then 
$I \cap J = \emptyset$.
\end{cor}

We now provide the description of the admissible polytopes of dimension $n-1$.

\begin{thm}\label{admn-1}
The admissible polytopes for $\Delta _{n,2}$ of dimension $n-1$ are  given by $\Delta _{n,2}$ and the    intersections  with  $\Delta _{n,2}$ of all  collections  of the  half-spaces of the form
 \[
H_{S} :  \sum\limits_{i\in S}x_{i} \leq 1, \;\; S\subset  \{1,\ldots , n\},\;  \|S\|=k,  \; 2\leq k\leq n-2,
\]
such that  if  $H_{S_1}, H_{S_2}$ belongs to a  collection then $S_1\cap S_2 =\emptyset$.
\end{thm} 

It follows   that any  polytope  which can be obtained as  the intersection with $\Delta _{n,2}$ of  a half-space of the form $\sum\limits_{s\in S}x_{s}\leq 1$, where $ S\subset  \{1,\ldots ,n\}$, $\|S\|=k$  and  $2\leq k\leq n-2$ is an admissible $(n-1)$-dimensional  polytope. Also, any polytope which can be obtained  as the intersection with $\Delta _{n,2}$ of the intersection of half-spaces of   the form
$\sum\limits_{i\in I}x_{i}\leq 1$ and $\sum\limits_{j\in J}x_{j}\leq 1$ such that $I\cap J = \emptyset$  is an admissible polytope, where $\|I\|, \|J\|\geq 2$.  Continuing in this way we describe  all admissible polytopes of dimension $n-1$.

\begin{proof}
Any admissible polytope  $P_{\sigma}$ of dimension $n-1$ which is not $\Delta _{n,2}$ has a facet which intersects $\stackrel{\circ}{\Delta}_{n,2}$.

1) If $P_{\sigma}$ has exactly one facet which intersect $\stackrel{\circ}{\Delta}_{n,2}$ then according to Theorem~\ref{n-2}, we have that $P_{\sigma}$ is given by those points in $\Delta _{n,2}$ which  satisfy  one of the inequalities  $\sum\limits_{i\in I}x_{i}\leq 1$ or   $\sum\limits_{i\in I}x_{i}\geq  1$ for some $I \subset  \{1,\ldots, n\}$, $\| I\|=l$  and $2\leq l\leq [\frac{n}{2}]$. Since  the points from $\Delta_{n,2}$  satisfy   $\sum\limits_{i=1}^{n}x_i = 2$ it follows that the second  inequality is equivalent to $\sum\limits_{j\in \{1, \ldots , n\}\setminus I}x_{j}\leq 1$.

2) Assume that  $P_{\sigma}$ has two facets which intersect $\stackrel{\circ}{\Delta}_{n,2}$ and which are  given by
\begin{equation}\label{facets1}
\sum\limits_{i\in I}x_{i}=1 \; \text{and}\; \sum\limits_{j\in J}x_{j} =1,
\end{equation}
which is equivalent to 
\begin{equation}\label{facets2}
\sum\limits_{i\in \{1, \ldots ,n\}\setminus I}x_{i}=1 \; \text{and}\; \sum\limits_{j\in \{1,\ldots, n\}\setminus J}x_{j} =1,
\end{equation}
where $I, J\subset \{1,\ldots, n\}$, $\|I\|, \|J\|\geq 2$.
Therefore, we can always assume that   $P_{\sigma}$ is given by 
\begin{equation}\label{prvi}
\sum\limits_{i\in I}x_{i}\leq 1 \; \text{and}\; \sum\limits_{j\in J}x_{j}\leq 1,
\end{equation}
Then  Lemma~\ref{empty} implies that 
$I\cap J=\emptyset$.

Arguing by induction one can prove  in the same way the statement for an admissible polytope $P_{\sigma}$ with an arbitrary finite number of interior  facets.
\end{proof}

\begin{ex}
It follows from Theorem~\ref{admn-1} that the admissible polytopes for $G_{4,2}$ of dimension $3$ are $\Delta _{4,2}$ and the intersection of $\Delta _{4,2}$ with the half spaces $x_i+x_j\leq 1$, $1\leq i<l\leq 4$. There are ${4\choose 2}=6$ such polytopes.
\end{ex}

\begin{ex}
The admissible polytopes for $G_{5,2}$  of dimension $4$ are: $\Delta _{5,2}$ and the polytope obtained by the intersection of $\Delta _{5,2}$ with the \begin{enumerate}
\item half spaces $x_i+x_j\leq 1$, $1\leq i<j\leq 5$
\item half space $x_i+x_j+x_k\leq 1$, $1\leq i<j<k\leq 5$, which can be written as $x_p+x_q\geq 1$, $1\leq p\leq q\leq 5$
\item intersection of the half spaces $x_i+x_j\leq 1$ and $x_p+x_q\leq 1$, where $\{i,j\}\cap \{p,q\}=\emptyset$, $1\leq i<j\leq 5$, $1\leq p<q\leq 5$.
\end{enumerate}
There are $10$ polytopes of  type (1) and any of them has $9$ vertices,  there are  $10$ polytopes of  type (2) and any of them has $7$ vertices and there are $15$ polytopes of type (3) and any of them has $8$ vertices.  
\end{ex}

\begin{ex}
The admissible polytopes for $G_{6,2}$  of dimension $5$ are: $\Delta _{6,2}$ and the polytope obtained by the intersection of $\Delta _{6,2}$ with the 
\begin{enumerate}
\item half spaces $x_i+x_j\leq 1$, $1\leq i<j\leq 6$;
\item half spaces $x_i+x_j+x_k\leq 1$, $1\leq i<j<k\leq 6$;
\item half spaces $x_i+x_j+x_k+x_l\leq 1$, $1\leq i<j<k<l\leq 6$ which can be written as $x_p+x_q\geq 1$, $1\leq p< q\leq 6$;
\item intersection of the half spaces $x_i+x_j\leq 1$ and $x_p+x_q\leq 1 $, where $\{i,j\}\cap \{p, q\}=\emptyset$, $1\leq i<j\leq 6$, $1\leq p<q\leq 6$.
\item intersection of the half spaces $x_i+x_j\leq 1$ and $x_p+x_q+x_s\leq 1$, where $\{i,j\}\cap\{p,q,s\} = \emptyset$, $1\leq i<j\leq 6$, $1\leq p<q<s\leq 6$.
\end{enumerate}

The number of these polytopes and their vertices are as follows: type (1)  -  $(15, 14)$, type (2) -  $(20, 12)$, type (3) -  $(15, 9)$, type (4) - $(45, 13)$ and type (5) - $(60, 11)$.    
\end{ex}

\section{Spaces  of parameters for $G_{n,2}$}

The algebraic torus $(\C ^{\ast})^{n}$ acts canonically on $G_{n,2}$ and its action is an extension of $T^n$- action. The strata $W_{\sigma}$ are invariant under $(\C ^{\ast})^{n}$-action and the corresponding orbit spaces $F_{\sigma} = W_{\sigma}/(\C ^{\ast})^{n}$ we call the spaces of parameters of the strata. 
In this section we discuss the spaces of parameters of  the strata in $G_{n,2}$ whose admissible polytopes intersect the  interior of  $\Delta _{n,2}$.  Moreover, we show  that the space $\mathcal{F}_{n}$ described in~\cite{BT-U}  is a universal space of parameters for $G_{n,2}$ and describe the virtual  spaces of parameters of these  strata. The notions of universal space of parameters an virtual spaces of parameters are  defined in~\cite{BT-2} and~\cite{BT-2nk}.

\subsection{The spaces of parameters for the strata in $G_{n,2}$}
Let us fix  the chart  $M_{12}$, this can be done without loss of generality because of the action of symmetric group $S_{n}$.  The elements of this chart can ne  represented by  the matrices which have the first $2\times 2$-minor equal to identity matrix. We consider the coordinates for the points from $M_{12}$  as the columns of the corresponding matrices that is
$(z_3,\ldots , z_{n}, w_{3}, \ldots, w_{n})$.  The charts are invariant for the action of the algebraic torus  $(\C ^{*})^{n}$ and  this action is,  in the local  coordinates of the chart $M_{12}$,  given by  
\begin{equation}\label{action}
(\frac{t_3}{t_1}z_3, \ldots , \frac{t_n}{t_1}z_{n},\frac{t_3}{t_2}w_{3}, \ldots , \frac{w_n}{t_2}w_{n}).
\end{equation}
If we put $\tau _{1} = \frac{t_3}{t_1},\ldots ,\tau _{n-2} =\frac{t_n}{t_1}, \tau _{n-1} = \frac{t_3}{t_2}$, we obtain  that
\[
\frac{t_{i}}{t_2} = \frac{\tau _{i-2}\cdot \tau _{n-1}}{\tau _{1}}, \;\; 4\leq i\leq n.
\]

\subsubsection{ The main stratum}\label{submain}

The main stratum $W$ is characterized by the condition that its points have non-zero  Pl\"ucker coordinates, it belongs to any chart and   its  admissible polytope is $\Delta _{n,2}$

. The  main stratum can be,  in the chart $M_{12}$, written   by the system of equations
\begin{equation}\label{orbit}
c_{ij}^{'}z_iw_{j} = c_{ij}z_iw_{j}, \;\; 3\leq i<j\leq n,
\end{equation} 
where the parameters $(c_{ij}^{'}:c_{ij})\in \C P^1$  and 
$c_{ij}, c_{ij}^{'}\neq 0$ and $c_{ij}\neq c_{ij}^{'}$ for all $3\leq i<j\leq n$.  

The number of parameters is $N={n-2 \choose 2}$ and it follows from~\eqref{orbit} that these parameters satisfy the following equations
\begin{equation}\label{relat}
 c_{3i}^{'}c_{3j}c_{ij}^{'} = c_{3i}c_{3j}^{'}c_{ij}, \;\; 4\leq i<j\leq n.
\end{equation}
The number of these equations is $M={n-3\choose 2}$. . From~\eqref{relat} we obtain that 
\begin{equation}\label{relatmain1}
(c_{ij}:c_{ij}^{'}) = (c_{3i}^{'}c_{3j}: c_{3i}c_{3j}^{'}), \;\; 4\leq i<j\leq n.
\end{equation}
\begin{equation}\label{relatmain2}
 (c_{3i}:c_{3i}^{'})\neq (c_{3j}:c_{3j}^{'}), \;\;  4\leq i<j\leq n.
\end{equation}

Note that~\eqref{relat} gives the embedding of the space $F_{n}=W/(\C ^{\ast})^{n}$ in $(\C P^{1})^{N}$, $N={n-2\choose 2}$.
Moreover,  it follows from~\eqref{relat} that $F_{n}$ is an open algebraic manifold in $(\C P^{1})^{N}$ given by the intersection of the cubic hypersurfaces~\eqref{relat} and  the conditions that $(c_{ij}^{'}:c_{ij})\in \C P^1\setminus A$. The dimension of $F_{n}$ is $2(n-3)$ what is exactly equal to $2(N-M)$.

\subsubsection{ An arbitrary stratum}\label{subarb}

Any stratum $W_{\sigma}$ consists of those points from $G_{n,2}$  for which some fixed Pl\"ucker coordinates are zero. If $W_{\sigma}\subset M_{12}$  then the stratum  $W_{\sigma}$ is defined by the condition $P^{1i}=0, P^{2j}=0$ and $P^{pq}=0$, for some   $3\leq i, j\leq n$ ,  $3\leq p<q\leq n$. In the local coordinates of the chart $M_{12}$ this condition translates to $w_{i}=z_{j}=0$ and $z_{p}w_{q} = z_{q}w_{p}$. Therefore, according to Lemma 14.10  from~\cite{BT-2nk}  any stratum $W_{\sigma}\subset M_{12}$ is obtained  by restricting  the  surfaces~\eqref{orbit} to some  $\C ^{J}$, where $J\subset \{(3,3),(3,4), \ldots , (n-1, n), (n,n)\}$ and $\| J\| =l$ for some $0\leq l\leq Q$, where $Q = (n- 2)^2$. 
In particular, if an admissible polytope for $W_{\sigma}$ has a maximal dimension $n-1$ we have  that $l\geq n-1$. It implies that  if the space of parameters $F_{\sigma} = W_{\sigma}/(\C ^{\ast})^{n}$ for $W_{\sigma}$ is not a point, then  it can be  obtained by restricting the intersection of the  cubic hypersurfaces~\eqref{relat} to some $q$  factors  $\C P^{1}_{B}=\C P^{1}\setminus B$ in $(\C P^{1})^{N}$, where $B=\{(1:0), (0:1)\}$  and $0\leq q\leq l$.

\begin{prop}\label{point}
If the admissible polytope $P_{\sigma}$ of  a stratum $W_{\sigma}$ intersects the  interior of $\Delta _{n,2}$  and $\dim P_{\sigma} =n-2$ then the space of parameters $F_{\sigma}$ is a point.
\end{prop}
\begin{proof}
 Since $P_{\sigma}$ is an interior polytope in $\Delta _{n,2}$  it follows from Lemma~\ref{polytopbound} that  there must exist  $z_{i}, w_{j}\neq 0$ for some $3\leq i,j\leq n$.   Because of the action of the symmetric group we can assume that $z_{i}\neq 0$ for  $3\leq i\leq l\leq n$. Note that  $l$ is $\geq 3$, but it must be that  $l\leq n-1$. This is because  $\dim P_{\sigma}=n-2$ , so  the stabilizer for $T^n$ action on $W_{\sigma}$ is of dimension $2$, that is $T^{n-2}$ is a maximal torus which acts freely on $W_{\sigma}$. Then we have that $w_{j}\neq 0$ for $l+1\leq j\leq n$ and $w_{j}=0$ for $3\leq j\leq l$. It implies that in the chart $M_{12}$ the stratum $W_{\sigma}$ is given by the coordinates $(z_{3}, \ldots z_{l}, 0, \ldots , 0, w_{l+1}, \ldots , w_{n})$, which means that $W_{\sigma}$ consists of one $(\C ^{*})^{n}$-orbit, that is $F_{\sigma}$ is a point.
\end{proof}

\subsection{An universal space of parameters for $G_{n,2}$}
The universal space of parameters $\mathcal{F}_{n}$ for a $(2n,k)$- manifold $M^{2n}$ with an effective action of the compact torus $T^k$, $k\leq n$, is a compactification of the space of parameters $F_{n}$ of the main stratum, see~\cite{BT-2nk}.
The  universal space of parameters  for the Grassmannian $G_{5,2}$ regarding to the canonical action of $T^5$  is defined and explicitly described in~\cite{BT-2}. The general notion of universal space of parameters is axiomatized in~\cite{BT-2nk}. In~\cite{klem} it is proved that the Chow quotient  $G_{n,2}\!/\!/ (\C ^{\ast})^{n}$ as defined in~\cite{Kap}  provides an  universal space  of parameters for $G_{n,2}$ regarded to the canonical $T^n$ -action.    The method  in~\cite{klem}  is based on the embedding   of the space  $F_n$  in $\C P^{N}$, $N = {n+1\choose 4}$ via the cross ratio of the Pl\"ucker coordinates.   In this paper we show   that the space $\mathcal{F}_{n}$  described in~\cite{BT-U}  is a universal space of parameters for $T^n$-action on $G_{n,2}$. This space
$\mathcal{F}_{n}$  is obtained in~\cite{BT-U} by the techniques of a wonderful compactification of the space of parameters $F_n$ of the main stratum, starting form the requirement that the homeomorphisms of $F_n$ which are defined by the transition functions between the standard Pl\"ucker charts for $G_{n,2}$ extends to the homeomorphisms for $\mathcal{F}_{n}$.  

Let us fix the  chart $M_{12}$ and consider a stratum $W_{\sigma}\subset M_{12}$. As already remarked, in the local coordinates for the chart $M_{12}$ this stratum is defined by the conditions $z_{s}=w_{m}=0$ and $z_{i}w_{j} = z_{j}w_{i}$ for some   $3\leq s, m\leq n$ and  $3\leq i < j\leq n$. On the other hand the main stratum $W$  is in the chart $M_{12}$ given by~\eqref{relat} and it is a dense  set in $G_{n,2}$. Using these two facts, one can  try to assign the new space of parameters $\tilde{F}_{\sigma}$ to $W_{\sigma}$.  The question which arises here is in which ambient space that is a  compactification $\mathcal{F}$  for $F$ the corresponding  assignment is to be done.  The  answer to this question  is determined by the condition that  such  assignment must not depend on a fixed chart.  

First it is obvious that $\mathcal{F}_{n}$ has to contain  the closure $\bar{F}_{n}$ of $F_n$ in $(\C P^{1})^{N}$. This closure is given by the intersection~\eqref{relat} in    $(\C P^{1})^{N}$. Second, the transition functions  between the charts produce  the homeomorphisms between the records of the space $F_{n}$ in these  charts. The compactification $\mathcal{F}_{n}$ should be such that these homeomorphisms extend to the homeomorphisms of  $\mathcal{F}_{n}$ for  an arbitrary two charts. 

Starting from these requirements the following theorem is proved   in~\cite{BT-U}:

\begin{thm}\label{universal}
Let the  manifold  $\mathcal{F}_{n}$ is obtained by the  wonderful compactification of $\bar{F}_{n}$  with the generating set  of subvarieties  $\bar{F}_{I}\subset \bar{F}_{n}$,  defined by $(c_{ik}:c_{ik}^{'}) = (c_{il} :c_{il}^{'}) = (c_{kl}:c_{kl}^{'}) = (1:1)$ for $ikl \in I$,  where $I \subset \{ikl \; | \: 3\leq i<k<l\leq n\}$. Then any homeomorphism $f_{ij, kl} : F_{n}\to F_{n}$ induced by the transition function  between the charts $M_{ij}$ and $M_{kl}$ extends to the homeomorphism of $\mathcal{F}_{n}$.
\end{thm}

We can now proceed as in the case of $G_{5,2}$. To any stratum  $W_{\sigma}$ we assign the virtual space of parameters $\tilde{F}_{\sigma , 12}\subset \mathcal{F}_{n}$ in the chart $M_{12}$ using the fact that the main stratum $W$  is a dense set in $G_{n,2}$ . In this regard we differentiate the following cases.

\begin{enumerate}
\item $W_{\sigma}\subset M_{12}$ -  we assign  $\tilde{F}_{\sigma , 12} \subset \mathcal{F}_{n}$ using the description~\eqref{orbit} of the main stratum.
\item $W_{\sigma}\cap M_{12}=\emptyset$ - we consider a chart $M_{ij}$ such that $W_{\sigma}\subset M_{ij}$ and assign to $W_{\sigma}$ the space $\tilde{F}_{\sigma , ij}$ using~\eqref{orbit}.  Then using  homeomorphism for $\mathcal{F}_{n}$ stated by Theorem~\ref{universal} we assign to $W_{\sigma}$  the subset $\tilde{F}_{\sigma, 12}\subset \mathcal{F}_{n}$ which is the image of   $\tilde{F}_{\sigma , ij}\mathcal{F}_{n}$ by this homeomorphism.   Proposition 9.11
from~\cite{BT-2}  applies here directly to show  that $\tilde{F}_{\sigma, 12}$ does not depend on the choice of a chart $M_{ij}$ which contains $W_{\sigma}$. 
\end{enumerate}

\begin{rem}
By the action of the symmetric group $S_n$ the same construction holds for an arbitrary chart $M_{ij}$, that is to any stratum $W_{\sigma}$ one can assign $\tilde{F}_{\sigma, ij}\subset \mathcal{F}_{n}$.
\end{rem}

\begin{rem}\label{revir}
It follows from~\eqref{orbit} that for  the main stratum it holds $\tilde{F}_{ij} \cong F$. Moreover, if the stratum $W_{\sigma}$ is defined by the condition that is has exactly one zero Pl\"ucker coordinate it also  follows from~\eqref{orbit}  that $\tilde{F}_{\sigma, ij} \cong F_{\sigma}$. This
is also true for the strata which have exactly two zero Pl\"ucker coordinates $P^{ij}$ and $P^{kl}$ where $i,j\neq k,l$. For all other strata it follows from~\eqref{orbit} that the corresponding  virtual  and real spaces of parameters are not homeomorphic, that is the virtual spaces of parameters are "bigger". For example, the space of parameters of the  stratum in $G_{n,2}$ defined by $P^{13}=P^{14}=P^{34}=0$ is, by the discussion preceding  Proposition~\ref{point}  homeomorphic to $(\C P^{1}_{A})^{\frac{(n-4)(n-5)}{2}}$ subject to  the   relations of the type~\eqref{relat}. Its   universal space of parameters is by~\eqref{orbit} homeomorphic to $\C P^{1}\times (1:0)^{2n-8}\times (\C P^{1}_{A})^{\frac{(n-4)(n-5)}{2}}$ subject to the relations~\eqref{relat}.
\end{rem}

It si straightforward to verify that  for an arbitrary chart $M_{ij}$ it  holds
\begin{equation}\label{union}
\bigcup\limits_{\sigma}\tilde{F}_{\sigma, ij}=\mathcal{F}_{n}.
\end{equation}

\begin{rem}
To illustrate~\eqref{union}  we follows~\cite{BT-2} and consider the Grassmann manifold $G_{5,2}$ and  fix  the chart $M_{12}$. 
 In the procedure of the wonderful compactification  $\mathcal{F}_{5}$  is obtained  by the blowing up of $\bar{F}_{5}$ at the point   $S= ((1:1), (1:1), (1:1))$.  Now,  in the chart $M_{12}$ the space of parameters of the stratum  $W_{\sigma}$ defined by  $P^{34}=P^{35}=P^{45}=0$ and $P^{ij}\neq 0$ for all others $i, j$,  is the point $S$.  The virtual space of parameters for $W_{\sigma}$ we can explicitly  obtain if we look at  this  stratum  in the chart $M_{13}$. Namely, the local coordinates  for  $W_{\sigma}$ in $M_{13}$ are $z_2, z_4= z_5=0, w_2, w_4, w_{5}$.  Therefore, it follows from~\eqref{orbit} that $\tilde{F}_{\sigma, 13}  = (1:0)\times (1:0)\times \C P^1 \cong \C P ^1$, which implies that $\tilde{F}_{\sigma, 12}\cong \C P^1$, that is $\tilde{F}_{\sigma, 12}$ is homeomorphic to the exceptional divisor at the point $S$.
\end{rem}

In addition,  in an analogous way as it  was done for $G_{5,2}$ in~\cite{BT-2},  paragraph 9.2, it can be proved:
\begin{prop}
There exists a canonical projection $g_{\sigma, ij} : \tilde{F}_{\sigma, ij} \to F_{\sigma}$ for any admissible set $\sigma$ and any chart $M_{ij}$.
\end{prop}
\begin{rem}
From  now on we will omit in the notation of the virtual spaces of parameters the indices of the  charts and write jut $\tilde{F}_{\sigma}$.
\end{rem}

Altogether,  the proof of Theorem 11.11 for $G_{5,2}$ directly generalizes to   $G_{n,2}$ for $n\geq 6$, which verifies that the third condition of Axiom 6 for $(2n,k)$-manifolds from~\cite{BT-2nk}  is satisfied. In this way we obtain:

\begin{thm}
The  space  $\mathcal{F}_{n}$ is the universal space of parameters for $G_{n,2}$.  
\end{thm}


\begin{ex}
 For $n=5$ the universal space of parameters  $\mathcal{F}_{5}$  is obtained as  the blow up of the smooth algebraic manifold $\bar{F}_{5}=\{((c_{34}:c_{34}^{'}), (c_{35}:c_{35}^{'}), (c_{45}:c_{45}^{'}))\in (\C P^{1})^{3} | c_{34}^{'}c_{35}c_{45}^{'} = c_{34}c_{35}^{'}c_{45}\}$ at the point $((1:1), (1:1), (1:1))$. This result is obtained in~\cite{BT-2}, see also~\cite{BT-U} for more general insight.
\end{ex}

\begin{ex}
For $n=6$ the universal space of parameters $\mathcal{F}_{6}$ is obtained in~\cite{BT-U}  by the wonderful compactification of the algebraic manifold  $\bar{F}_{6}\subset \in (\C P^{1})^{6}$ defined by 
\[c_{34}^{'}c_{35}c_{45}^{'} = c_{34}c_{35}^{'}c_{45},\;  c_{34}^{'}c_{36}c_{46}^{'} = c_{34}c_{36}^{'}c_{46},
\]
\[  c_{35}^{'}c_{36}c_{56}^{'} = c_{35}c_{36}^{'}c_{56},\;  c_{45}^{'}c_{46}c_{56}^{'} = c_{45}c_{46}^{'}c_{56},
\]
with generating subvarieties 
\[
\bar{F}_{345}=\bar{F} \cap \{ (c_{34} : c_{34}^{'}) = (c_{35} : c_{35}^{'}) = (c_{45}: c_{45}^{'})=(1:1)\}
\]
\[
\bar{F}_{346}=\bar{F} \cap \{(c_{34}:c_{34}^{'})= (c_{36}: c_{36}^{'}) = (c_{46}:c_{46}^{'})= (1:1)\}
\]
\[
\bar{F}_{356}=\bar{F} \cap \{(c_{35}:c_{35}^{'}) =  (c_{36}:c_{36}^{'})= (c_{46}: c_{46}^{'})= (1:1))\}
 \]
\[
\bar{F}_{456} = \bar{F}\cap \{(c_{45}:c_{45}^{'}) = (c_{46}:c_{46}^{'})= (c_{56}:c_{56}^{'})= (1:1)\}.
\]
\end{ex}

\subsection{Critical and singular points for $G_{n,2}/T^n$}

The  notion of a critical point on $G_{n,2}/T^n$ can be defined in one of the equivalent ways. Consider a moment map $\mu _{n,2} : G_{n,2} \to \Delta _{n,2}$. It is a smooth map and there are, in the standard way, defined the critical points and the critical values of this map. It is proved in~\cite{BT-2nk} that a point $L\in G_{n,2}$ is a critical point of the moment map $\mu _{n,2}$ if and only if the stabilizer of $L$ related to the standard $T^n$-action on $G_{n,2}$ is non-trivial. This is equivalent to say that  the admissible polytope of the stratum which contains $L$ is not of maximal dimension $n-1$.

We say that a point $[L]\in G_{n,2}/T^n$ is a critical point if $L\in G_{n,2}$ is a critical point in the above sense.  This is correctly defined since obviously the notion of a  critical point on $G_{n,2}$  is invariant for $T^n$-action. 

The notion of the critical points in $G_{n,2}/T^n$ can be as well related to the  singularities  of $G_{n,2}/T^n$  following tubular neighborhood theorem. Precisely, the tubular neighborhood theorem states  that for any point $L\in G_{n,2}$  there exists $T^n$-equivariant diffeomorphism between the vector bundle $T^n \times _{T_{L}}V$ and the neighborhood of the orbit $T^n\cdot L$ in $G_{n,2}$, where $T_{L}$ is the stabilizer of the point $L$ and $V$ is the normal bundle in $(TG_{n,2})_{T^n\cdot L}$ to the tangent bundle $T (T^n\cdot L)$.  It implies that there exists a neighborhood   in the orbit space $G_{n,2}/T^n$ of the point defined by the orbit $T^n\cdot L$  which has the form $(T^n\times _{T_{L}}V)/T^n = V^{T_{L}}\times \text{cone}(S(U)/T_{L})$, where $V^{T_{L}}$ is the subspace of $V$ consisting of the vectors fixed by $T_{L}$, $U$ is subspace of $V$ defined by $V = V^{T_{L}}\oplus U$ related to some $T^n$- invariant metric on $G_{n,2}$ and $S(U)$ is the corresponding unit sphere.  In this way one concludes that any  point in $G_{n,2}/T^n$ which is  defined by a point from $G_{n,2}$  having  the non-trivial  stabilizer,  has a neighborhood with cone-like singularities. In this way all singularities of the orbit space 
$G_{4,2}/T^4\cong S^5$ are described in~\cite{BT-1}.

On the other hand to any stratum $W_{\sigma}\subset G_{n,2}$ we assigned the corresponding space of parameters $F_{\sigma}$ and the virtual space of parameters $\tilde{F}_{\sigma}$.  As we noted in Remark~\ref{revir} these spaces are in general not homeomorphic.  We say that a point $L \in G_{n,2}$ is a singular point for $T^n$-action on $G_{n,2}$ if the space of parameters of the stratum $W_{\sigma}$ such that $L\in W_{\sigma}$ is not homeomorphic to the virtual space of parameters for $W_{\sigma}$.  Since the  notions of spaces of parameters and virtual spaces of parameters are obviously invariant for the standard $T^n$-action we can define  the notion of a singular point in the orbit space $G_{n,2}/T^n$.

\begin{defn}
A point $[L]= T^n \cdot L\in G_{n,2}/T^n$ is said to be a singular point for the standard $T^n$-action on $G_{n,2}$ if  the space of parameters $F_{\sigma}$ of the stratum $W_{\sigma}$, $L\in W_{\sigma}$ is not homeomorphic to the virtual space parameters $\tilde{F}_{\sigma}$ for $W_{\sigma}$.
\end{defn}

The singular points can be characterized in terms of the Pl\"ucker coordinates as follows.

\begin{prop}\label{crpl}
A point $[L]\in G_{n,2}/T^n$ is a singular point if and only if there exists $i$, $1\leq i\leq n$  such that $P^{ij}(L)=0$ for all $j\neq i$, $1\leq j\leq n$ or there exist $1\leq i<j<k\leq n$ such that $P^{ij}(L)=P^{ik}(L) = P^{jk}(L)=0$.
\end{prop}

\begin{proof}
We first note that all points from  from $T^n\cdot L\subset G_{n,2}$ have the same Pl\"ucker coordinates. Let  $[L]$ is a singular point  in $G_{n,2}/T^n$ and let $L\in W_{\sigma}$. If  $P_{\sigma}\subset \partial \Delta _{n,2}$ then by  Lemma~\ref{polytopbound} we have that $P_{\sigma}$  belongs to the plane $x_i=0$ or to the plane $x_i=1$ for some $1\leq i\leq n$.  It implies that $P^{ij}(L)=0$ for all $j\neq i$ or $P^{jk}(L)=0$ for all $j,k\neq i$. Thus, the point $L$ satisfies the condition of the statement. If $P_{\sigma}\cap \stackrel{\circ}{\Delta}_{n,2}\neq \emptyset$,  assume that  the vertex $\Lambda _{12} \in P_{\sigma}$. Then $W_{\sigma}$ belongs to the chart $M_{12}$ and  let $z_3, \ldots, z_n, w_3, \ldots, w_n$ be the local  coordinates in this chart. Since $F_{\sigma}$ is not homeomorphic to  $\tilde{F}_{\sigma}$ and $P_{\sigma}\not \subset \partial \Delta _{n,2}$ there exist $3\leq i<j\leq n$ such that $z_i=z_j=0$ or $w_i=w_j=0$. It implies that $P^{2i}(L) = P^{2j}(l)=P^{ij}(L)=0$ or $P^{1i}(L)=P^{1j}(L)=P^{ij}(L)=0$,  that is the statement holds. 

In proving opposite direction we can assume as well that $W_{\sigma}\subset M_{12}$ and that $i=1$. Then $j, k\geq 3$ and in the local coordinates for $M_{12}$ we have that $w_j=w_k=0$. Using~\eqref{orbit} it implies that $F_{\sigma}$ and $\tilde{F}_{\sigma}$ are not homeomorphic, they differ at least by $\C P^{1}$.
\end{proof}    

\begin{prop}\label{crsing}
All critical points for $G_{n,2}/T^n$ are for $n\geq 5$ the singular points  for $G_{n,2}/T^n$.
\end{prop}

\begin{proof}
Let $[L]$ be a critical point in $G_{n,2}/T^n$, assume that $L$ belongs to the chart $M_{12}$ and let $z_3, \ldots, z_n, w_3, \ldots ,w_n$ be the local coordinates  in this chart. Since the stabilizer for $L$ is non-trivial and $n\geq 5$, it follows from~\eqref{action}
that there must exist $3\leq i<j\leq n$ such that $z_i=z_j=0$ or $w_i=w_j=0$ or there must exist $3\leq i\leq n$ such that $z_i=w_i=0$. In other words there exist   $3\leq i<j\leq n$ such that $P^{2i}(L)=P^{2j}(L) = P^{ij}(L)=0$ or $P^{1i}(L)=P^{1j}(L) = P^{ij}(L)=0$ or there exists $3\leq i\leq n$ such that $P^{ij}(L)=0$ for any $j\neq i$. Then Proposition~\ref{crpl} implies that the point $[ L]$ is a singular point in $G_{n,2}/T^n$.
\end{proof}

\begin{rem}
Proposition~\ref{crsing} does not hold for $n=4$. In that case the points which in that chart $M_{12}$ have  coordinates $z_3, z_4=0, w_3=0, w_4$  or $z_3=0, z_4, w_3, w_4=0$  represent the points from $G_{4,2}/T^4$ which are  critical, but which are   not  singular points.  The points which have the local coordinates of this form in some chart exhaust all critical points in $G_{4,2}/T^4$ which are not singular.
\end{rem}
 
Let $\text{Sing}X_n$ denotes the set of singular point in $X_{n}$ and put $Y_{n}= X_{n}\setminus \text{Sing}X_n$.

\begin{prop}
The set $Y_{n}\subset X_n$ is a open, dense set in $X_n$, which is a manifold.
\end{prop}

\begin{proof}
The boundary $\partial W_{\sigma} =  \overline{W}_{\sigma} \setminus W_{\sigma}$ is by~\cite{GS} the union of the strata $W_{\sigma}^{'}$ which correspond to the faces $P_{\sigma ^{'}}$ of the admissible polytope $P_{\sigma}$, for any stratum $W_{\sigma}$.  It implies  by~\ref{subarb} that $F_{\sigma ^{'}}\subseteq  F_{\sigma}$  and  by~\ref{submain} that $\tilde{F}_{\sigma}\subseteq \tilde{F}_{\sigma ^{'}}$,  for any such $\sigma ^{'}$. Therefore,  if a stratum $W_{\sigma}$  consists of singular points then $\overline{W}_{\sigma}$ consists of  singular points as well. It implies that $\text{Sing}X_{n} = \cup \overline{W}_{\sigma}$, where the union goes over all the strata consisting of singular points, so $X_n$ is  a closed set.  Thus,  $Y_{n}$ is an open, dense set in $X_n$, which is a manifold, as it contains the orbit space of the main stratum and it does not contain the critical points.  
\end{proof}

\section{The virtual spaces of parameters for $G_{n,2}$}
We establish now some properties of the virtual space of parameters for $G_{n,2}$ which turn out to be important for the description of the orbit space $G_{n,2}/T^n$.

Let $W_{\sigma}$ be a stratum such that it admissible polytope satisfies    $P_{\sigma}\cap \stackrel{\circ}{\Delta}_{n,2}\neq \emptyset$. Then $\dim P_{\sigma}= n-2$ or $\dim P_{\sigma}= n-1$. If $\dim P_{\sigma}=n-2$ it is proved in the previous sections that $W_{\sigma}$ is an one-orbit stratum while for $\dim P_{\sigma}=n-1$ it is not  in general.

\begin{thm}\label{univ-(n-2)}
If $P_{\sigma}$ is an admissible polytope such that $\dim P_{\sigma}= n-1$ and $P_{\sigma ^{'}}$ is a facet of $P_{\sigma}$ such that $P_{\sigma ^{'}}\cap \stackrel{\circ}{\Delta}_{n,2}\neq \emptyset$ then
\[
\tilde{F}_{\sigma}\subseteq \tilde{F_{\sigma ^{'}}},
\]
where $\tilde{F}_{\sigma}$ and $\tilde{F}_{\sigma ^{'}}$ are the virtual spaces of parameters for $W_{\sigma}$ and $W_{\sigma ^{'}}$.
\end{thm}

\begin{proof}
We first note that $W_{\sigma ^{'}}\subset W_{\sigma}$ since $W_{\sigma ^{'}}$ is an one-orbit stratum. Assume without loss of generality that $\Lambda _{12}$ is a common vertex for $P_{\sigma}$ and $P_{\sigma ^{'}}$. Let $P_{\sigma^{'}}$ be given by the intersection of $x_{i_1}+\ldots +x_{i_l} =1$ with $\Delta _{n,2}$, where $1\leq i_1<i_2\ldots <i_l\leq [\frac{n}{2}]$. Since $\Lambda _{12}$ is a vertex for $P_{\sigma ^{'}}$, because of the action of the symmetric group,  we can assume that this plane is given by $x_1+x_3+\ldots +x_l=1$. The stratum $W_{\sigma ^{'}}$ belongs to the chart $M_{12}$ and in this chart it writes as
\[
\left(
\begin{array}{cc}
1 & 0 \\
0 & 1 \\
a_3 & 0 \\
\vdots & \vdots \\
a_l & 0\\
0 & b_{l+1} \\
\vdots & \vdots \\
0 & b_n 
\end{array}
\right ),
\]
where $a_3,\ldots , a_l\neq 0$ and  $b_{l+1},\ldots , b_{n}\neq 0$ as  $\dim P_{\sigma}=n-2$. This is because in~\cite{BT-2}  we proved  that the dimension of the maximal subtorus which acts freely on $W_{\sigma ^{'}}$ is equal to the $\dim P_{\sigma}=n-2$.  It follows that
\[
\tilde{F}_{\sigma ^{'}} = (\C P^{1})^{n-3}\times (1:0)^{n-l}\times (\C P^{1})^{l-4}\times (1:0)^{n-l}\times \cdots
\]
\[\cdots  \times \C P^{1}\times (1:0)^{n-l}\times (1:0)^{n-l}\times (\C P^{1})^{\frac{(n-l)(n-l-1)}{2}},
\]
subject to the relations  $c_{ij}c_{ik}^{'}c_{jk} = c_{ij}^{'}c_{ik}c_{jk}^{'}$, where $1\leq i<j<k\leq n$ and $(c_{ij}:c_{ij}^{'})\in \C P^1$.

Since $P_{\sigma^{'}}$ is  a face of $P_{\sigma}$  then the points from $P_{\sigma}$ satisfy  either $x_1+ x_3+ \ldots +x_l\leq 1$ either $x_1+x_3+\ldots +x_l\geq 1$.  It follows that the stratum $W_{\sigma}$ belongs to the chart $M_{12}$ as well.

If it holds $x_1+x_3+\ldots +x_l\leq 1$ then the stratum $W_{\sigma}$ is given by
\[
\left(
\begin{array}{cc}
1 & 0 \\
0 & 1 \\
a_3 & 0 \\
a_4& \vdots \\
a_l & 0\\
a_{l+1} & b_{l+1} \\
\vdots & \vdots \\
a_n & b_n 
\end{array}
\right ),
\]
where $a_3, \ldots, a_l\neq 0$ and $b_{l+1}, \ldots, b_n\neq 0$ since $P_{\sigma ^{'}}$ is a facet of $P_{\sigma}$.  Moreover, some of $a_{l+1}, \ldots , a_{n}$ are  non-zero since $T^{n-1}$ acts freely on $W_{\sigma}$. Therefore,
\[
\tilde{F}_{\sigma} = (\C P^{1})^{l-3}\times (1:0)^{n-l}\times (\C P^{1})^{l-4}\times (1:0)^{n-l}\times \cdots
\]
\[\cdots  \times \C P^{1}\times (1:0)^{n-l}\times (1:0)^{n-l}\times \mathcal{D},
\]
where $\mathcal{D}\subset   (\C P^{1})^{\frac{(n-l)(n-l-1)}{2}}$  is  a proper subset,  
subject to the relations  $c_{ij}c_{ik}^{'}c_{jk} = c_{ij}^{'}c_{ik}c_{jk}^{'}$, where $1\leq i<j<k\leq n$. It follows that $\tilde{F}_{\sigma}\subset \tilde{F}_{\sigma ^{'}}$.

If it holds $x_1+x_3+\ldots +x_l\geq 1$ the stratum $W_{\sigma}$ is given by
\[
\left(
\begin{array}{cc}
1 & 0 \\
0 & 1 \\
a_3 & b_3 \\
a_4& \vdots \\
a_l & b_l\\
0 & b_{l+1} \\
\vdots & \vdots \\
0  & b_n 
\end{array}
\right ),
\]
where $a_3, \ldots, a_l\neq 0$ and $b_{l+1}, \ldots, b_n\neq 0$ since $P_{\sigma ^{'}}$ is a facet of $P_{\sigma}$.  Also some of  $b_{3}, \ldots, b_l$  are  non-zero,  since  $T^{n-1}$ acts freely on $W_{\sigma}$.  Therefore,
\[
\tilde{F}_{\sigma} =  \mathcal{D}_1\times (1:0)^{n-l}\times \mathcal{D}_2\times (1:0)^{n-l}\times \cdots
\]
\[\cdots  \times \mathcal{D}_{l-3}\times (1:0)^{n-l}\times (1:0)^{n-l}\times (\C P^{1})^{\frac{(n-l)(n-l-1)}{2}},
\]
where $\mathcal{D}_1\subset (\C P^{1})^{l-3}$, $\mathcal{D}_2\subset (\C P^{1})^{l-4}$, ..., $\mathcal{D}_{l-3}\subset \C P^1$ and  $\mathcal{D}_{1}$  is a proper subset.
It follows that $\tilde{F}_{\sigma}\subset \tilde{F}_{\sigma ^{'}}$.
\end{proof}

From the above proof we directly deduce:

\begin{cor}
Let $P_{\sigma ^{'}}$ be an admissible polytope such that $\dim P_{\sigma ^{'}}=n-2$ and $P_{\sigma ^{'}}\cap \stackrel{\circ}{\Delta}_{n,2}\neq \emptyset$. Then
\begin{equation}
\tilde{F}_{\sigma ^{'}} = \bigcup\limits_{\sigma} \tilde{F}_{\sigma},
\end{equation}
where the union goes over all $\sigma$ such that $P_{\sigma ^{'}}$ is a facet of $P_{\sigma}$.
\end{cor}

In an analogous way  the same statement can be proved for an admissible polytope $P_{\sigma ^{'}}$ such that $\dim P_{\sigma ^{'}} = n-2$ and $P_{\sigma ^{'}}\subset \partial \Delta _{n,2}$, as any such polytope is defined by an additional condition $x_i=0$ or $x_i=1$ for some $1\leq i\leq n$.  Altogether we have:

\begin{prop}
  Let $P_{\sigma ^{'}}$ be an admissible polytope such that $\dim P_{\sigma ^{'}}=n-2$. Then 
\begin{equation}
\tilde{F}_{\sigma ^{'}} = \bigcup\limits_{\sigma} \tilde{F}_{\sigma},
\end{equation}
where the union goes over all $\sigma$ such that $P_{\sigma ^{'}}$ is a facet of $P_{\sigma}$.
\end{prop}

Since the boundary $\partial \Delta _{n,2}$ consists of hypersimplices $\Delta _{n-1, 2}$ and simplices $\Delta _{n-1,1}$ which correspond in $G_{n,2}$  to the Grassmannians  $G_{n-1,2}$ and complex projective spaces $\C P^{n-2}$,   altogether  way we obtain:
\begin{thm}
The universal space of parameters $\mathcal{F}_{n}$ for $T^n$-action on $G_{n,2}$ is given by the formal  union
\begin{equation}
\mathcal{F}_{n} =\underset{\underset{ P_{\sigma}\cap \stackrel{\circ}{\Delta}_{n,2}\neq \emptyset}{\dim P_{\sigma}=n-2}}{\bigcup} \tilde{F}_{\sigma}.
\end{equation}
\end{thm}

\section{Universal space of parameters and admissible polytopes}
 Let  $\mathcal{F}$ be a universal space of parameters for $G_{n,2}$ and $\tilde{F}_{\sigma}$ the  virtual space of parameters for a stratum $W_{\sigma}$. For $x\in \stackrel{\circ}{\Delta}_{n,2}$ denote by
\[
\tilde{F}_{x} =  \underset{x\in \stackrel{\circ}{P}_{\sigma}}{\bigcup}\tilde{F}_{\sigma}.
\]

\begin{thm}
$\tilde{F}_{x} =\mathcal{F}_{n}$ for any $x\in \stackrel{\circ}{\Delta}_{n,2}$.
\end{thm}

 In order to prove this theorem, by   Theorem~\ref{univ-(n-2)}, we need to prove that $\tilde{F}_{\sigma}\subset \tilde{F}_{x}$  for any $P_{\sigma}$ such that $\dim P_{\sigma} = n-2$ and $P_{\sigma}\cap \stackrel{\circ}{\Delta}_{n,2}\neq \emptyset$.

\subsection{Proof for $G_{5,2}$}
It follows from Theorem~\ref{univ-(n-2)} that the admissible polytopes of dimension $3$ are given by the intersection  with $\Delta _{5,2}$ of the planes $x_i+x_j=1$, where $1\leq i<j\leq 5$. Without loss of generality because of the action of the symmetric group $S_5$, fix the admissible polytope defined  by $x_1+x_4=1$, denote it by $P_{14}$, the corresponding stratum by $W_{14}$ and  and its virtual space of parameters by $\tilde{F}_{14}$. The stratum $W_{14}$ belongs to the chart $M_{12}$  and it writes as
\[
\left (
\begin{array}{cc}
1 & 0 \\
0 & 1 \\
0 & b_3 \\
a_4 & 0 \\
0 & b_5
\end{array}
\right ), \;\;\;  a_4, b_3, b_5\neq 0.
\]
It follows that $\tilde{F}_{14} \cong (1:0)\times \C P^{1}\times (1:0)$. 

Let $x\in \stackrel{\circ}{\Delta}_{5,2}$. The following cases are  possible:

\begin{enumerate}
\item If $x\in \stackrel{\circ}{P}_{14}$  then $\tilde{F}_{14}\subset \tilde{F}_{x}$. 
\item  If $x\notin \stackrel{\circ}{P}_{14}$ then  $x_1+x_4>1$ or $x_1+x_4<1$.
\end{enumerate}
If $x_1+x_4>1$ then $x$ belongs to the admissible polytope, we denote it by $P_{14}^{+}$ defined by the half-space $x_1+x_4 \geq 1$. The stratum  which corresponds to this  polytope is, in the chart $M_{12}$, given by
\[
\left (
\begin{array}{cc}
1 & 0 \\
0 & 1 \\
0 & b_3 \\
a_4 & b_4 \\
0 & b_5
\end{array}
\right ),\;\;\; a_4, b_3,  b_4, b_5\neq 0,   
\]
so the corresponding virtual space of parameters is $\tilde{F}_{14}^{+} \cong (0:1)\times \C P^1\times (1:0)$. Since $\tilde{F}_{14}^{+}=\tilde{F}_{14}$, it follows that $\tilde{F}_{14}\subset \tilde{F}_{x}$.

If $x_1+x_4<1$ then $x$ belongs to the admissible polytope, we denote it by $P_{14}^{-}$ defined by the half-space $x_1+x_4 \leq 1$. The corresponding stratum is, in the chart $M_{12}$, given by
\[
\left (
\begin{array}{cc}
1 & 0 \\
0 & 1 \\
a_3 & b_3 \\
a_4 & 0 \\
a_5 & b_5
\end{array}
\right ), \;\;\; a_3,a_4,a_5, b_3, b_5\neq 0.
\]
The corresponding space of parameters is $\tilde{F}_{14}^{-} \cong (0:1)\times \C P^{1}_{A}\times (1:0)$ and we have that $\tilde{F}_{14}^{-}\subset \tilde{F}_{x}$. We need to prove that the points $A_{1}=(0:1)\times (0:1)\times (1:0)$, $A_{2} =(0:1)\times (1:0)\times (1:0)$ and $A_3=(0:1)\times (1:1)\times (1:0)$ belong to $\tilde{F}_{x}$.  In  order to do that we consider the following cases.

1) 
 a) If $x_2+x_3\leq 1$ then $x$ belongs to the intersection of half-spaces $x_1+x_4\leq 1$ and $x_2+x_3\leq 1$. Consider the stratum given by
\[
\left (
\begin{array}{cc}
1 & 0 \\
0 & 1 \\
0 & b_3 \\
a_4 & 0 \\
a_5 & b_5
\end{array}
\right ), \;\;\; a_4, a_5, b_3, b_5\neq 0.
\]
Its admissible polytope is exactly given  by the intersection of these half-spaces  and its virtual space of parameters coincides with its space of parameters which is the point $A_1$. Thus, in this case $A_1\in \tilde{F}_{x}$.

b) If $x_2+x_3\geq 1$  then $x_1+x_4+x_5\leq 1$ and we consider the stratum  given by
\[
\left (
\begin{array}{cc}
1 & 0 \\
0 & 1 \\
a_3 & b_3 \\
a_4 & 0 \\
a_5 & 0
\end{array}
\right ), \;\;\; a_3,a_4,a_5, b_3\neq 0.
\]
Its admissible polytope is given by $x_2+x_3\geq 1$,  its virtual space of parameters is $(0:1)\times (0:1)\times \C P^{1}$  and it contains the point $A_{1}$. It follows that $A_1\in \tilde{F}_{x}$.

2) a) If $x_2+x_5\leq 1$ the  point $x$ belongs to the intersection of half-spaces  $x_1+x_4\leq 1$ and $x_2+x_5\leq 1$. Now consider the stratum
\[
\left (
\begin{array}{cc}
1 & 0 \\
0 & 1 \\
a_3 & b_3 \\
a_4 & 0 \\
0 & b_5
\end{array}
\right ), \;\;\; a_3, a_4, b_3, b_5\neq 0.
\]
It admissible polytope is exactly given by the intersection of these half-spaces and its virtual space of parameters coincides with its 
space of parameters that is with the point $A_{2}=(0:1)\times (1:0)\times (1:0)$. Thus, $A_{2}\in \tilde{F}_{x}$.

b) if $x_2+x_5\geq 1$ it follows that $x_1+x_3+x_4\leq 1$, so we consider the stratum
\[
\left (
\begin{array}{cc}
1 & 0 \\
0 & 1 \\
a_3 & 0 \\
a_4 & 0 \\
a_5 & b_5
\end{array}
\right ), \;\;\; a_3, a_4, a_5, b_5\neq 0.
\]
Its admissible polytope is exactly given by $x_2+x_5\geq 1$ and its virtual space of parameters is $\C P^{1}\times (1:0)\times (1:0)$.
So, $A_{2}\in \tilde{F}_{x}$.

3)
a) If $x_3+x_5\leq 1$ then $x$ belongs to the intersection of the half-spaces  $x_1+x_4\leq 1$ and $x_3+x_5\leq 1$ and we consider the stratum
\[
\left (
\begin{array}{cc}
1 & 0 \\
0 & 1 \\
a_3 & b_3 \\
a_4 & 0 \\
a_5 & b_5
\end{array}
\right ),\;\;\;  a_3, a_4, a_5, b_3, b_5\neq 0, \; a_3b_5=a_5b_3.
\]
Its admissible polytope is exactly given by this intersection. Its virtual space of parameters  coincides with its space of parameters and it is the point $A_3= (0:1)\times (1:1)\times (1:0)$. Thus, $A_3\in \tilde{F}_{x}$.

b) If $x_3+x_5\geq 1$ the $x_1+x_2+x_3\leq 1$. The stratum $W_{35}^{+}$ whose admissible polytope is defined in this way does not belong to the chart $M_{12}$. This stratum belongs to the chart $M_{13}$ and in this charts it writes as
\[
\left (
\begin{array}{cc}
1 & 0 \\
a_2 & 0 \\
0 & 1 \\
a_4 & 0 \\
a_5 & b_5
\end{array}
\right ), a_2, a_4, a_5, b_5\neq 0.
\]
Its virtual space of parameters is $\C P^1\times (1:0)\times (1:0)$. The transition function  between  the charts $M_{12}$ and $M_{13}$ induces the homeomorphism $f : F_{12}\to F_{13}$ which can be extended to the homeomorphism $\tilde{f} : \mathcal{F}_{5} \to \mathcal{F}_{5}$. This homeomorphism is given by
\[
((c_{34}:c_{34}^{'}), (c_{35}:c_{35}^{'}), (c_{45}:c_{45}^{'})) \to
\]
\[ ((c_{34}:c_{34}-c_{34}^{'}), (c_{35}:c_{35}-c_{35}^{'}),( c_{35}(c_{34}-c_{34}^{'}) : c_{34}(c_{35}-c_{35}^{'})).
\]
It implies that the preimage of $\C P^1\times (1:0)\times (1:0)$ is given by $((c_{34}:c_{34}^{'}), (1:1), (c_{34}:c_{34}^{'}))$, where
$(c_{34}:c_{34}^{'})\in \C P^1$. Thus, for $c_{34}=0$ we obtain that the point $A_{3} = (0:1)\times (1:1)\times (1:0)$ belongs to $\tilde{F}_{x}$.

Altogether we proved in this way that $\tilde{F}_{14}\subset \tilde{F}_{x}$.

By the action of the symmetric group we obtain this to be true for all $\tilde{F}_{ij}$, that is
\[
\tilde{F}_{x} = \mathcal{F}_{5}.
\]

Using the same pattern the  following holds as well:

\begin{prop}\label{emptyint}
If $P_{\sigma}, P_{\sigma ^{'}}$ are admissible polytopes such that $\stackrel{\circ}{P}_{\sigma}\cap \stackrel{\circ}{P}_{\sigma ^{'}}$ has  non-empty intersection with $\stackrel{\circ}{\Delta}_{n,2}$  then $\tilde{F}_{\sigma}\cap \tilde{F}_{\sigma ^{'}}=\emptyset$.
\end{prop}

\begin{proof}
We differentiate   the following cases depending on the dimensions of $P_{\sigma}$ and $P_{\sigma ^{'}}$.

1.  Let $\dim P_{\sigma}=\dim P_{\sigma ^{'}} = n-1$.  We provide the proof for the case when  each of $P_{\sigma}$ and $P_{\sigma ^{'}}$  is defined by just  one half-space according to Theorem~\ref{admn-1}. In the case when $P_{\sigma}$ or $P_{\sigma ^{'}}$ is given as the intersection of the larger number of half spaces the proof goes in an analogous way.
Now,  because of the action of the symmetric group we can assume that $P_{\sigma}$ is defined by the  half-space $x_1+x_2+\ldots +x_k\leq 1$, $2\leq k\leq n-2$ and let $P_{\sigma}^{'}$ is defined by $x_{p_1}+\ldots +x_{p_s}\leq 1$. Since $x_1+\ldots +x_n=2$ and these two half-spaces intersect it follows that there exists $i$, $k+1\leq i\leq n$ such that $i\neq p_1, \ldots p_s$. Without loss of generality we can assume that $i=k+1$.  

 It follows that the strata $W_{\sigma}$ and $W_{\sigma ^{'}}$  belong to the chart $M_{1,k+1}$. The stratum $W_{\sigma}$ writes in this charts as
\[
\left(
\begin{array}{cc}
1 & 0 \\
a_2 & 0 \\
a_3  & 0 \\
\vdots & \vdots \\
a_k & 0\\
0 & 1\\
a_{k+2} & b_{k+2} \\
\vdots & \vdots \\
a_n  & b_n 
\end{array}
\right ),
\]
where $a_i\neq 0$, $2\leq i\leq n$, $i\neq k+1$ and $b_i\neq 0$, $k+2\leq i\leq n$. It follows that
\begin{equation}\label{virt1}
\tilde{F}_{\sigma} =  (\C P^{1})^{k-2}\times (1:0)^{n-k-1}\times (\C P^{1})^{k-3}\times (1:0)^{n-k-1}\times \ldots \times \C P^{1}\times (1:0)^{n-k-1}
\end{equation}
\[
\times (1:0)^{n-k-1}\times (\C P^{1}_{A})^{\frac{(n-k-1)(n-k-2)}{2}}.
\]
The stratum $W_{\sigma ^{'}}$ is, in this chart,  given by the conditions $b_{p_s}^{'} =0$, $1\leq s\leq l$.  Now if $\tilde{F}_{\sigma}\cap \tilde{F}_{\sigma}^{'}\neq \emptyset$  than  in $\tilde{F}_{\sigma ^{'}}$ we have $(c_{pq}: c_{pq}^{'}) = (1:0)$  or  $(c_{pq}: c_{pq}^{'}) = \C P^1$   for any  $2\leq p\leq k$ and any $k+2\leq q \leq n$.  In both cases we must have that $b_{p}^{'}=0$ which implies that $\{2, \ldots , k\}\subset \{p_1, \ldots , p_l\}$.   Moreover, there must exist $p_s\notin \{2,\ldots , k\}$  and note that not all $b_{i}^{'}$-s are zeros.  We can assume that $b_{k+2}^{'}\neq 0$ and $b_{k+3}^{'}=0$.   It follows that in $\tilde{F}_{{\sigma}^{'}}$ we have that $(c_{k+2, k+3}:  c_{k+2, k+3}^{'}) = (0:1)$   which  together with~\eqref{virt1} gives the  contradiction with an assumption that $\tilde{F}_{\sigma}$ and $\tilde{F}_{\sigma ^{'}}$ intersect.

2. Let $\dim P_{\sigma} = n-1$ and $\dim P_{\sigma^{'}}=n-2$. Assume that $P_{\sigma}$ is given by $x_1+\ldots +x_k\leq 1$ , $2\leq k\leq n-2$ and  $P_{\sigma ^{'}}$ is given by $x_{p_1}+\ldots+x_{p_l}=1$, $2\leq l\leq n-2$.  Since these polytopes intersect  it follows that we can assume that  $k+1\notin\{p_1, \ldots, p_l\}$ and that $p_1\notin \{1, \ldots k\}$.  It follows that the both strata $W_{\sigma}$ and $W_{\sigma ^{'}}$ belong to the chart $M_{k+1, p_1}$. Then stratum  $W_{\sigma}$ is given in this chart by
\[
\left(
\begin{array}{cc}
a_1 & b_1 \\
a_2 & b_2 \\
\vdots & \vdots \\
a_k & b_k\\
1 & 0\\
a_{k+2} & b_{k+2} \\
\vdots & \vdots \\
0  & 1\\
a_{p_1+1} & b_{p_1+1}\\
\vdots & \vdots\\
a_n & b_n 
\end{array}
\right ),
\]
where $a_ib_i=a_jb_i$ for $1\leq i<j \leq k$ and $a_i, b_i\neq 0$ for $i\geq k+2$. It follows that
\begin{equation}\label{virt2}
\tilde{F}_{\sigma} = (1:1)^{k-1}\times (\C P^{1}_{A})^{n-k-2}\times (1:1)^{k-2}\times (\C P^{1}_{A})^{n-k-2}\times \ldots \times (1:1)\times (\C P^{1}_{A})^{n-k-2}
\end{equation}
\[
\times (\C P^{1}_{A})^{n-k-2} \times ((\C P^{1}_{A})^{n-k-2})^{\frac{(n-k-1)(n-k-2)}{2}}.
\]
The  points of the stratum  $W_{\sigma ^{'}}$ in this chart satisfy conditions $a_{p_s}^{'} = 0$ for $1\leq s\leq l$.  It follows that $b_{p_s}^{'}\neq 0$ and $b_{q}^{'}=0$ for $q\neq p_s$, $1\leq s\leq l$, see Proposition~\ref{point}. Note that there exists
$q_{0}\neq k+1$ such that $q_{0}\notin \{p_1, \ldots , p_l\}$.  It follows that $(c_{q_{0}p_s}: c_{q_{0}p_s}^{'}) = (0: 1)$ or $(1:0)$ in $\tilde{F}_{\sigma}^{'}$. Comparing to~\eqref{virt2} we conclude that $\tilde{F}_{\sigma}\cap \tilde{F}_{\sigma ^{'}}=\emptyset$.

3.  Let $\dim P_{\sigma} = \dim P_{\sigma^{'}}=n-2$ and assume that $P_{\sigma}$ is given by $x_1+\ldots +x_k=1$, $2\leq k\leq n-2$  and $P_{\sigma ^{'}}$ is given by $x_{p_1}+\ldots +x_{p_l}=1$, $2\leq l\leq n-2$. We can assume that $x_1\notin \{p_1, \ldots, p_l\}$ and $p_1\notin \{1, \ldots, k\}$. Then the strata $W_{\sigma}$ and $W_{\sigma ^{'}}$ belong to the chart $M_{1p_1}$   The stratum $W_{\sigma}$ is in this chart given by  
\[
\left(
\begin{array}{cc}
1 & 0 \\
a_2 & 0 \\
\vdots & \vdots \\
a_k & 0\\
0 & b_{k+1}\\
\vdots & \vdots \\
0 & 1\\
0 & b_{p_1+1} \\
\vdots & \vdots \\
0 & b_n 
\end{array}
\right ),
\]

wher $b_{i}\neq  0$, $2\leq i\leq k$, $b_{j}\neq 0$, $k+1\leq i<j \leq n$, $ j\neq p_1$. It follows that 
\begin{equation}\label{virt3}
\tilde{F}_{\sigma}= (\C P^{1})^{k-2}\times (1:0)^{n-k-1}\times \ldots \times \C P^{1}\times (1:0)^{n-k-1}
\end{equation}
\[
\times (1:0)^{n-k-1}\times (\C P^{1})^{\frac{(n-k-1)(n-k-2)}{2}}.
\]
The points of the stratum $W_{\sigma ^{'}}$ are in this chart satisfy  the conditions $a_{p_s}^{'}=0$ for all $2\leq s\leq l$. It follows that $b_{p_s}^{'}\neq 0$ and $b_{q}^{'}=0$ for $q\neq p_s$ , $2\leq s\leq l$.  
Let $q_{0}$ be such that  $q_{0}\notin\{1, \ldots, k\}$  and $q_{0}\notin \{p_1, \ldots, p_{l}\}$.  Note that there exist $i_{0}$, $2\leq i_0\leq k$ such that $i_{0}\in \{ p_1, \ldots p_l\}$.  It  implies that  $(c_{i_{0}q_{0}}: c_{i_{0}q_{0}}^{'})= (0:1)$  in $\tilde{F}_{\sigma}^{'}$ which together with~\eqref{virt3} implies that $\tilde{F}_{\sigma}\cap \tilde{F}_{\sigma ^{'}}=\emptyset$.  
\end{proof}

\begin{cor}
Let $x\in \stackrel{\circ}{\Delta}_{n,2}$. If $\tilde{F}_{\sigma}, \tilde{F}_{\sigma ^{'}}\subset \tilde{F}_{x}$ then $\tilde{F}_{\sigma}\cap \tilde{F}_{\sigma ^{'}}=\emptyset$.
\end{cor}

\section{The chamber decomposition for $\Delta _{n,2}$}
Consider the hyperplane arrangement in $\R ^{n}$ given by
\begin{equation}\label{hyp}
\mathcal{A}_{n} : \;  \Pi \cup \{x_i=0, 1\leq i\leq n\}\cup \{x_i=1, 1\leq i\leq n\},
\end{equation}
where the set $\Pi$ is given by~\eqref{novo}, and  the face  lattice $L(\mathcal{A}_{n})$ of  the    hyperplane arrangement $\mathcal{A}_{n}$ .   This lattice consists  of the  hyperplanes from $\mathcal{A}_{n}$ and all intersections of the elements from $\mathcal{A}_{n}$.  

 The hyperplane arrangement $\mathcal{A}_{n}$ induces  the hyperplane arrangement in  $\R ^{n-1} = \{ (x_1, \ldots, x_n)\in \R ^{n} : x_1+\ldots +x_n=2\}$ which is obtained by intersecting this $\R ^{n-1}$ with the planes from~\eqref{hyp}. 

Denote by $L(\mathcal{A}_{n,2}) = L(\mathcal{A}_{n})\cap \Delta _{n,2}$. Then $L(\mathcal{A}_{n,2})$ provides decomposition for $\Delta _{n,2}$ which we call  chamber decomposition and for an element $C\in L(\mathcal{A}_{n,2})$ we say to be a  chamber.

\begin{lem}
Let  $C\in L(\mathcal{A}_{n,2})$ such that $\dim C=n-1$.  If  $C$ has a nonempty intersection with   $\stackrel{\circ}{P}_{\sigma}$ then $C\subset \stackrel{\circ}{P}_{\sigma}$.  Thus, $C$ can be  be obtained as the intersection of  the interiors of all admissible polytopes which contain it.
\end{lem}
\begin{proof}
 We first note that if $\stackrel{\circ}{P}_{\sigma}\cap C \neq \emptyset$ then $\dim P_{\sigma}=n-1$ as well. Moreover, any facet of $P_{\sigma}$ which intersects $\stackrel{\circ}{\Delta} _{n,2}$ belongs to   some of the hyperplanes from the set $\Pi$  which define the chamber decomposition $\mathcal{C}$. For the second statement, we note that any wall $V$ of the chamber $C$ is contained in a facet of an  admissible polytope which contains the chamber $C$. This follows from the fact that any of the  hyperplanes  from $\Pi$  divides the hypersimplex $\Delta _{n,2}$ into two admissible polytopes.  
\end{proof}
The same is true for an arbitrary  chamber from $L(\mathcal{A}_{n,2})$.
\begin{lem}
Any element $C\in L(\mathcal{A}_{n,2})$ can be obtained as the intersection of  the interiors all admissible polytopes which contain $C$.
\end{lem}

\begin{proof}
We have that  $L(\mathcal{A}_{n,2})= L(\mathcal{A}) \cap \Delta _{n,2}$. So, if $C\in L(\mathcal{A}_{n,2})$ it follows that  $C= \cap \pi _{i_1, \ldots i_p}$  for some planes  $\pi _{i_1, \ldots i_p}\in \Pi$.  Since   $\pi _{i_1, \ldots , i_p}\cap \stackrel{\circ}{\Delta} _{n,2}$ is an admissible polytope and the planes from $\Pi$  define  the chamber decomposition $L(\mathcal{A}_{n,2})$ the statement follows.
\end{proof}

Therefore we have:
\begin{prop}\label{chpol}
The chamber decomposition $L(\mathcal{A}_{n,2})$ coincides with the decomposition of $\stackrel{\circ}{\Delta} _{n,2}$ given by the intersections of interiors  all  admissible  polytopes which are inside  $\Delta _{n,2}$ . 
\end{prop}

Further on we denote  the chambers from $L(\mathcal{A}_{n,2})$  by $C_{\omega}$, where $\omega$ consist of those admissible sets $\sigma$ such that $C_{\omega}\subset \stackrel{\circ}{P}_{\sigma}$.

\subsection{Moment map and the chamber decomposition}

Let $C_{\omega} \in L(\mathcal{A}_{n,2})$ be a chamber, that is $\dim C_{\omega}=n-1$.  According to Proposition~\ref{chpol}  we have that $C_{\omega} = \cap _{\sigma \in \omega}\stackrel{\circ}{P}_{\sigma}$  and  obviously $\dim P_{\sigma}= n-1$.  It follows from~\cite{BT-2} that the spaces 
$\hat{\mu} ^{-1}(x) = \cup _{\sigma \in \omega} F_{\sigma} \subset G_{n,2}/T^n $ are smooth manifolds and  they are diffeomorphic   for all $x\in C_{\omega}$, that is diffeomorphic to some   manifold $F_{\omega}$. 

 On the other hand, as it is showed in~\cite{BT-2} there exist the  canonical homeomorphisms $h_{\sigma} :  W_{\sigma}/T^{\sigma}\to \stackrel{\circ}{P}_{\sigma}\times F_{\sigma}$  given by $h_{\sigma} = (\hat{\mu}_{\sigma}, p_{\sigma})$,
where $\hat{\mu}_{\sigma} : W_{\sigma}/T^{\sigma} \to \stackrel{\circ}{P}_{\sigma}$ is induced by the moment map $\hat{\mu} : G_{n,2}/T^n \to \Delta _{n,2}$, while $p_{\sigma} : W_{\sigma}/T^{\sigma} \to F_{\sigma}$ is induced by the natural projection $G_{n,2} \to G_{n,2}/(\C ^{\ast})^{n}$.  For the main stratum $W$ we have $W/T^{n}\cong \Delta _{n,2}\times F$, so  it follows that  $F\subset F_{\omega}$.  Let $\hat{C}_{\omega} = \hat{\mu} ^{-1}(C_{\omega})$.
\begin{cor}
For any  $C_{\omega}\in L(\mathcal{A}_{n,2})$ such that  $\dim C_{\omega}=n-1$ there exists canonical homeomorphism
\[
h_{C_{\omega}} : \hat{C}_{\omega} \to   C_{\omega}\times F_{\omega}.
\]
where  the manifold  $F_{C\omega}$ is a compactification of the space $F_{n}$  given by the   spaces $F_{\sigma}$ such that $C_{\omega}\subset P_{\sigma}$
\begin{equation}\label{FC}
F_{\omega}=\bigcup\limits_{C_{\omega}\subset \stackrel{\circ}{P}_{\sigma}} F_{\sigma}.
\end{equation}

\end{cor}

In some sense this can be generalized to  an arbitrary  chamber.  
As noted in~\cite{GM}, perceiving   the spaces $\hat{\mu}^{-1}(x)$,  $\hat{\mu}^{-1}(y)$  as symplectic quotients it follows that they are homeomorphic, that is to some space $F_{\omega}$  , for any chamber $C_{\omega} \in L(\mathcal{A}_{n,2})$ and any two points $x, y\in C$.  When $\dim C_{\omega}=n-2$  we deduce the following:

\begin{lem}
For any  $C_{\omega}\in L(\mathcal{A}_{n,2})$, $\dim C_{\omega}=n-2$ there exists canonical homeomorphism
\[
h_{C_{\omega}} : \hat{C}_{\omega}\to  C_{\omega}\times F_{\omega}.
\]
The space $F_{\omega}$  is a compactification of the space  $F$  and this compactification is given by the spaces $F_{\sigma}$ such that $C_{\omega}\subset  \stackrel{\circ}{P}_{\sigma}$ and  $\dim P_{\sigma}=n-1$  and  a point, that is the  space $F_{\sigma}$ such that  $C_{\omega} \subset  \stackrel{\circ}{P}_{\sigma}$ and $\dim P_{\sigma} =n-2$.  
\end{lem}
\begin{proof}
It is obvious that   if $C_{\omega} \subset  \stackrel{\circ}{P}_{\sigma}$ and $\dim P_{\sigma}=n-1$ then $F_{\sigma}\subset F_{\omega}$.   Since $C_{\omega}$ is of dimension $n-2$  there exists unique  admissible polytope $P_{\sigma}$, $\dim P_{\sigma}=n-2$ such that $C_{\omega} \subset  \stackrel{\circ}{P}_{\sigma}$, in fact $P_{\sigma}$  is defined by the underlying hyperplane for $C_{\omega}$. Thus, using Proposition~\ref{point}  the statement follows.
\end{proof}  

Since the only interior admissible polytopes are of dimension $n-1$ or $n-2$, ,  for an arbitrary  $C_{\omega}\in L(\mathcal{A}_{n,2})$ i of dimension  $\leq n-3$ we  repeat the argument and  in an analogous way  deduce the following:

\begin{lem}
For any  $C_{\omega}\in L(\mathcal{A}_{n,2})$, $\dim C_{\omega}\leq n-3$ there exists canonical homeomorphism
\[
h_{C_{\omega}} : \hat{C}_{\omega}\to   C_{\omega}\times F_{\omega}.
\]
The space $F_{\omega}$  is a compactification of $F$ given by the spaces $F_{\sigma}$ such that  $C_{\omega} \subset  \stackrel{\circ}{P}_{\sigma}$ and $\dim P_{\sigma}=n-1$ and $q$ points   where $q\geq 2 $ is the number of polytopes $P_{\sigma}$ such that   $C_{\omega} \subset  \stackrel{\circ}{P}_{\sigma}$  and $\dim P_{\sigma}=n-2$.
\end{lem}

Note that the permutation action of the symmetric group $S_{n}$  on $\R ^{n}$ induces  $S_n$-action on $\mathcal{A}_{n}$, which further gives  $S_n$-action on the chambers    of the  chamber decomposition $L(\mathcal{A}_{n,2})$.  Together with Corollary~\ref{cormain} we obtain:
\begin{cor}
The spaces $F_{\omega}$ and  $\mathfrak{s}(F_{\omega})$ are homeomorphic for any $\mathfrak{s}\in S_{n}$ and any   $C_{\omega} \in  L(\mathcal{A}_{n,2})$.
\end{cor}

Altogether we conclude:

\begin{prop}
A manifold  $F_{\omega}$ is a compactification of the space $F\subset (\C P^{1}_{A})^{N}$  given by the equations~\eqref{relat}. This compactification  is given by the   spaces $F_{\sigma}$ such that $C_{\omega}\subset \stackrel{\circ}P_{\sigma}$. Moreover,  a space $F_{\sigma}$ is a point or it is obtained  by  restricting  the hypersurfaces~\eqref{relat} to some  factors  $(\C P^{1}_{B})^{q}\subset (\C P^{1}_{A})^{N} $, where $B=\{(1:0), (0:1)\}$  and $0\leq q\leq l$, $n-1\leq l\leq N$. 
\end{prop}

\subsection{Chamber decomposition and  virtual spaces of parameters}
Let $\mathcal{F}_{n}$ be an universal space of parameters for $G_{n,2}$ and   consider the chart $M_{12}$. We assigned to any   stratum  $W_{\sigma}$  from $G_{n,2}$ the virtual space of parameters  $\tilde{F}_{\sigma, 12}$  as described in~\cite{BT-2} and in the previous sections. Moreover, for any $\tilde{F}_{\sigma, 12}$ it is defined the projection $p_{12} : \tilde{F}_{\sigma, 12}\to F_{\sigma}$, where  $F_{\sigma}$ is the space of parameters for the stratum $W_{\sigma}$.


For $C_{\omega} \in L(\mathcal{A}_{n,2})$ from Proposition~\ref{emptyint} one directly deduces the following

\begin{cor}\label{cdis}
Let  $C_{\omega}\in L(\mathcal{A}_{n,2})$. Then $\tilde{F}_{\sigma}\cap \tilde{F}_{\bar{\sigma}}=\emptyset$ for any admissible sets $\sigma , \bar{\sigma}$ such that  $C_{\omega} \subset \stackrel{\circ}{P}_{\sigma}. \stackrel{\circ}{P}_{\bar{\sigma}}$.
\end{cor}

Recall that we defined the charts $M_{ij}$ for $G_{n,2}$ in~\ref{adm-n-2}. Together  with Theorem~\ref{main} we obtain:

\begin{cor}\label{ch-univ}
The union 
\begin{equation}
\mathcal{F}_{n} = \bigcup\limits_{C_{\omega}\subset \stackrel{\circ}{P}_{\sigma}}\tilde{F}_{\sigma}.
\end{equation}
is a disjoint union for any $C_{\omega}\in L(\mathcal{A}_{n,2})$.  
Therefore,  for any chamber $C_{\omega}$ and any chart $M_{ij}\subset G_{n,2}$   it is defined the projection $p_{C_{\omega}, ij} : \mathcal{F}_{n} \to F_{\omega}$ by $p_{C_{\omega}, ij}(y) = p_{\sigma, ij}(y)$. where $y\in \tilde{F}_{\sigma, ij}$. 
\end{cor}


\section{Summary - the orbit space $G_{n,2}/T^n$}

Let  $\stackrel{\circ}{G}_{n,2}/T^n = \hat{\mu}^{-1}(\stackrel{\circ}{\Delta}_{n,2})$  and let $\hat{C}_{\omega} = \hat{\mu}^{-1}(C_{\omega})$ be as before.  
\begin{thm}\label{jedan}
The following disjoint decomposition holds:
 \begin{equation}\label{dcd}
\stackrel{\circ}{G}_{n,2}/T^n\cong \bigcup\limits _{\omega} \hat{C} _{\omega}  \cong   \bigcup\limits_{\omega}(C_{\omega}\times F_{\omega}).
\end{equation}
where the topology on the right hand side is given by the induced moment map $\hat{\mu} : G_{n,2}/T^n \to \Delta _{n,2}$ and the natural projection $G_{n,2}/T^n \to G_{n,2}/(\C ^{\ast})^{n}$.
\end{thm}

The symmetric  group $S_n$ acts on $\stackrel{\circ}{G}_{n,2}/T^n$ by $\mathfrak{s}(C_{\omega}\times F_{\omega}) = \mathfrak{s}(C_{\omega})\times \mathfrak{s}(F_{\omega})$, that is $C_{\omega}\times F_{\omega}$ is homeomorphic to  $\mathfrak{s}(C_{\omega}\times F_{\omega})$, which significantly  simplifies the description  of the elements in the union~\eqref{dcd}. This means that there exists $l=l(n)\geq 4$ and  $\omega _1, \ldots , \omega _l$ - the indices of the $S_n$-action orbits representatives such that  
\begin{equation}\label{int}
\stackrel{\circ}{G}_{n,2}/T^n = \bigcup\limits _{i=1}^{l}\bigcup\limits_{\mathfrak{s}\in S_n}\mathfrak{s}(C_{\omega_{i}}\times F_{\omega _{i}}).
\end{equation}

\begin{ex}
For $n=4$, the chambers in $\stackrel{\circ}{\Delta}_{4,2}$ are of dimensions $0,1,2,3$ and  the $S_4$-action has one orbit in each of these dimensions, that is $l(4)=4$. 
\end{ex}

Altogether,
\begin{equation}\label{all-orb}
G_{n,2}/T^n \cong  \stackrel{\circ}{G}_{n,2}/T^n    \cup (\bigcup\limits_{q=1}^{n} G_{n-1,2}(q)/T^{n-1}) \cup (\bigcup\limits_{q=1}^{n}\Delta ^{n-2}(q)).
\end{equation} 
The topology on the right hand side of~\eqref{all-orb} is defined by the canonical embeddings $\C P^{n-2}(q)\to G_{n,2}$ and $G_{n-1, 2}(q)\to G_{n,2}$, $1\leq q\leq n$.

The universal space of parameters for a $\C P^{n-2}(q)$, $1\leq q\leq n$  is a point.
The    canonical embeddings  $i_{q} : G_{n-1,2}(q)\to G_{n,2}$  are defined by the  inclusions $\C ^{n-1}\to \C ^{n}$, $(z_1, \ldots , z_{n-1}) \to (z_1, \ldots z_{q-1}, 0, z_{q}, \ldots z_{n-1})$, $1\leq q\leq n$. Therefore,  it is straightforward  to relate the universal spaces of parameters for $G_{n,2}$ and $G_{n-1, 2}(q)$, $1\leq q\leq n$.

\begin{prop}\label{un-bound}
The universal space of parameters $\mathcal{F}_{n-1, q}$ for $G_{n-1,2}(q)\subset G_{n,2}$, $1\leq q\leq n$  can be obtained from the universal space of parameters $\mathcal{F}_{n}$ for $G_{n,2}$ by the restriction:
\begin{equation}\label{prq}
\mathcal{F}_{n-1, q} = {\mathcal{F}_{n}}|_{\{(c_{ij}:c_{ij}^{'}), i,j\neq q\}},
\end{equation}
which defines the projection $r_q:  \mathcal{F}_{n}\to \mathcal{F}_{n-1, q}$.
\end{prop}

It follows that  all previous constructions apply  to  $\mathcal{F}_{n-1,q}$   and $\Delta _{n-1,2}(q)\subset \partial \Delta _{n,2}$ obtained as $\Delta _{n-1,2}(q) = \Delta _{n,2} \cap \{ x_{q}=0\}$, $1\leq q\leq n$. Denote by  $p^{q}_{C_{\omega}, ij} : \mathcal{F}_{n-1,q} \to F_{\omega}$ the map given by Corollary~\ref{ch-univ}  for   the Grassmannian $G_{n-1,2}(q)$, $1\leq q\leq n$.   In this way  we obtain:

\begin{cor}\label{ind}
For any chamber $C_{\omega}\subset \partial \Delta _{n,2}$ and any chart $M_{ij}$ it is defined the projection $p_{C_{\omega}, ij} : \mathcal{F}_{n} \to F_{\omega}$. If $C_{\omega}\subset \Delta ^{n-1}(q)$ this  projection  maps $\mathcal{F}_{n}$ to a point, while  for  $C_{\omega}\subset \Delta _{n-1,2}(q)$ it is defined by  $p_{C_{\omega}, ij}(y) =( p^{q}_{C_{\omega}, ij}\circ r_q)(y)$.
\end{cor}

So, let us now consider the space
\begin{equation}
U_{n} = \Delta _{n,2}\times \mathcal{F}_{n}.
\end{equation}
 It can be  inductively defined the projection from $U_n$ using the following pattern:
\begin{equation}
U_{n} \to (\stackrel{\circ}{\Delta}_{n,2} \times \mathcal{F}_{n})\cup  (\bigcup\limits_{q=1}^{n}U_{n-1,q})\cup (\bigcup\limits_{q=1}^{n}\Delta ^{n-2}(q)),
\end{equation}
for  $U_{n-1, q} = \Delta _{n-1, 2}(q)\times \mathcal{F}_{n-1,q}$, which is an  identity for $x\in \stackrel{\circ}{\Delta}_{n,2}$, it is given  by  $(x, f)\to  (x, r_{q}(f))$ if  $x\in \Delta _{n-1,2}(q)$, while  $(x, f)\to x$ if  $x\in \Delta ^{n-2}(q)$, where  $1\leq q\leq n$. 


Therefore, from Corollary~\ref{ch-univ}, Theorem~\ref{jedan} and Corollary~\ref{ind} we obtain: 

\begin{thm}\label{prefin}
 For any chart $M_{ij}\subset G_{n,2}$  the map
\begin{equation}
G : U_{n} \to G_{n,2}/T^n, \;\; G (x, y) = h_{C_{\omega}}^{-1}(x, p_{C_{\omega}, ij}(y)) \;\; \text{if and only if}\;\;  x\in C_{\omega}
\end{equation}
is correctly defined.
\end{thm}

Then from Theorem~\ref{jedan}, formula~\eqref{all-orb} and Proposition~\ref{un-bound} we deduce:
\begin{thm}\label{fin}
The map $G$ is a continuous surjection and the orbit space $G_{n,2}/T^n$ is homeomorphic to the quotient of the space $U_{n}$ by the map $G$.
\end{thm}

The more explicit  proof proceeds in an analogous  way as the proof of Theorem 11.1  in~\cite{BT-2}, which describes  the orbit space $G_{5,2}/T^5$.

Victor M.~Buchstaber\\
Steklov Mathematical Institute, Russian Academy of Sciences\\ 
Gubkina Street 8, 119991 Moscow, Russia\\
E-mail: buchstab@mi.ras.ru
\\ \\ 

Svjetlana Terzi\'c \\
Faculty of Science and Mathematics, University of Montenegro\\
Dzordza Vasingtona bb, 81000 Podgorica, Montenegro\\
E-mail: sterzic@ucg.ac.me 

\end{document}